\documentclass{article}
\usepackage{geometry}

\usepackage{amsmath, amsthm, amsfonts, amssymb}
\usepackage{mathtools}
\usepackage{float}
\usepackage{caption}
\usepackage{subcaption}
\usepackage{booktabs, multirow, multicol}
\usepackage{indentfirst}
\usepackage{setspace}
\usepackage{listings}
\usepackage{bm}
\usepackage[dvipsnames]{xcolor}
\usepackage[ruled, lined, linesnumbered]{algorithm2e}
\usepackage{enumitem}

\usepackage{graphicx}
\graphicspath{{./figure/}}

\usepackage{epstopdf}

\usepackage[colorlinks=true,linkcolor=cyan,urlcolor=magenta,citecolor=violet]{hyperref}

\theoremstyle{plain}
\newtheorem{theorem}{Theorem}[section]

\newtheorem{lemma}[theorem]{Lemma}

\theoremstyle{definition}

\newtheorem{example}{Example}

\theoremstyle{remark}
\newtheorem{remark}{Remark}[section]

\newtheorem*{remark*}{Remark}
\newtheorem*{note*}{Note}

\DeclareMathOperator{\qr}{qr}

\DeclareMathOperator{\rank}{rank}
\DeclareMathOperator{\diag}{diag}
\DeclareMathOperator{\sgn}{sgn}
\DeclareMathOperator*{\argmax}{arg\,max}

\newcommand{\keywords}[1]{\vskip 2ex\par\noindent\normalfont{\bfseries Keywords: } #1}

\title{High-Order Implicit Low-Rank Method with Spectral Deferred Correction for Matrix Differential Equations}
\author{Shun Li
    \thanks{School of Mathematical Sciences, University of Science and Technology of China, Hefei, Anhui 230026, China.
    E-mail: {\tt lishun@mail.ustc.edu.cn}.}
\and Yan Jiang
    \thanks{School of Mathematical Sciences, University of Science and Technology of China, Hefei, Anhui 230026, China.
    Email: {\tt jiangy@ustc.edu.cn}.
    Research supported by NSFC grant 12271499.}
\and Yingda Cheng
    \thanks{Department of Mathematics, Virginia Tech,
    Blacksburg, VA 24061 U.S.A.
    E-mail:  {\tt yingda@vt.edu}.
    Research supported by DOE grant DE-SC0023164 and Virginia Tech.  }}
\date{}

\begin{document}

\maketitle

\begin{abstract}
    In this paper, we develop a low-rank method with high-order temporal accuracy using spectral deferred correction (SDC) to compute linear matrix differential equations.
    Low rank techniques exploit the fact that the solution to a differential equation can be approximated by a low-rank matrix, therefore, by only storing and operating on the low-rank factors, we can save computational cost and storage.
    The dynamic low-rank approximation (DLRA) is a well-known technique to realize such computational advantages for time-dependent problems.
    In~\cite{appelo2024robust}, a numerical method is proposed to correct the modeling error of the basis update and the Galerkin (BUG) method, which is a computational approach for DLRA.
    This method (merge-BUG/mBUG method) has been demonstrated to compute numerically convergent solutions for general advection-diffusion problems.
    However, similar to the original BUG method, it is only first-order accurate in time.

    On the other hand, SDC is a general framework to construct high-order time discretizations from low-order time discretizations.
    In this paper, we explore using SDC to elevate the convergence order of the mBUG method.
    In SDC, we start by computing a first-order solution by mBUG,
    and then perform successive updates by computing low-rank solutions to the Picard integral equation.
    Rather than a straightforward application of SDC with mBUG, we propose two aspects to improve computational efficiency.
    The first is to reduce the intermediate numerical rank by detailed analysis of dependence of truncation parameter on the correction levels.
    It turns out that  truncation tolerances depends on the correction levels, and larger tolerance can be used in the initial and early correction levels,
    which will reduce the numerical rank of the intermediate solutions.
    The second aspect is a careful choice of subspaces in the successive correction to avoid inverting large linear systems (from the K- and L-steps in BUG).
    We prove that the resulting scheme is high-order accurate for the Lipschitz continuous and bounded dynamical system.

    We further consider numerical rank control in our framework by comparing two low-rank truncation strategies:
    the hard truncation strategy by truncated singular value decomposition and the soft truncation strategy by soft thresholding.
    We demonstrate numerically that soft thresholding offers better rank control in particular for higher-order schemes for weakly (or non-)dissipative problems.
    Numerical results on benchmark tests are provided validating the performance of the method.
\end{abstract}


\section{Introduction}\label{sec:introduction}
This paper concerns the computation of low-rank solutions with high order implicit discretizations of linear matrix differential equations.
The model equation takes the form
\begin{equation}
    \label{eqn:mode}
    \frac{d}{dt}X(t)=F(X(t),t),  \quad X(t) \in \mathbb{R}^{m_1\times m_2}, \quad X(0)=X_0,
\end{equation}
where
\[
    F(X(t),t)=\sum_{j=1}^s A_j X(t) B_j^\top +G(t),
\]
and $A_j \in \mathbb{R}^{m_1\times m_1}, B_j \in \mathbb{R}^{m_2\times m_2}$ are sparse or structured matrices where a fast matrix-vector product is known.
Further, $G(t) \in  \mathbb{R}^{m_1\times m_1} $ is a given function that is assumed to have a known low-rank decomposition.
The integer $s$ is assumed not to be too large.
Equation \eqref{eqn:mode} is assumed to be stiff to necessitate implicit methods.
Such a model arises in many applications governed by diffusion type of partial differential equations (PDE) after spatial discretizations have been deployed.
In particular, we are interested in the case when $X(t)$ admits a low-rank approximation.
For such cases, low-rank methods will achieve a significant reduction in memory footprint and computational cost.

In recent years, the low-rank approach received increased attention in numerical analysis because it offers a way to tame  the curse of dimensionality,
enabling approximations to solutions of high dimensional PDEs \cite{hackbusch2012tensor, grasedyck2013literature, khoromskij2018tensor, bachmayr2023low}.
In two dimensions, this approach reduces to the low-rank matrix method.
A popular approach for computing low-rank solutions for time-dependent problems is the dynamic low-rank approximation (DLRA) \cite{koch2007dynamical} or Dirac–Frenkel time-dependent variational principle.
It was designed to capture a fixed rank solution by projecting the equation onto the tangent space of the fixed rank low-rank manifold.
In particular, the DLRA solves
\begin{equation}
    \label{eq:DLRAF}
    \frac{d}{dt}X(t)=\mathcal{P}_{X(t)} (F(X(t), t)),
\end{equation}
where $\mathcal{P}_{X(t)}$ is the orthogonal projection onto the tangent space $T_{X(t)}\mathcal{M}_r$ of the the rank $r$ matrix manifold
$\mathcal{M}_r=\{X \in  \mathbb{R}^{m_1\times m_2}, \rank(X)=r\}$ at $X(t)$.
The current state-of-the-art for DLRA is the basis update and Galerkin  (BUG) method \cite{ceruti2022unconventional} and its rank adaptive version \cite{ceruti2022rank}.
The BUG method performs well in many applications.
However, it suffers from the modeling error from the tangent projection $\|(\mathcal{I}-\mathcal{P}_{X(t)})F(X,t)\|$,
and this error can not be easily bounded or estimated \cite{appelo2024robust, lam2024randomized}.
In \cite{appelo2024robust}, an implicit Euler based BUG method with a simple modeling error fix by merging the BUG space with the numerical subspace spanned by  $F(X,t)$ was proposed.
We call this method merge-BUG (mBUG) method.
Numerical evidence suggested that this simple modification can yield a convergent scheme for a large class of equations.
For Lipschitz continuous and bounded dynamical system, we prove the stability and convergence of the mBUG scheme.
However, similar to the original BUG method, the mBUG method is only first-order accurate in time.

The main focus of this work is the extension of the mBUG method to higher-order temporal accuracy.
High order extensions of the BUG method have been considered in \cite{nakao2023reduced,ceruti2024robust},
where the key is the enrichment of the row and column subspaces to incorporate higher-order information.
While \cite{ceruti2024robust} focused on the explicit midpoint method, \cite{nakao2023reduced} considered diagonally implicit Runge-Kutta (DIRK) and implicit-explicit (IMEX) RK method. Other approaches with high order time integrators include low rank methods using Krylov-based techniques, see e.g. \cite{el2024krylov,meng2024preconditioning}.
In this work, rather than focusing on a specific type of temporal discretizations, we consider the spectral deferred correction (SDC) method \cite{dutt2000spectral},
which is a general framework to construct high order time discretizations from low-order time discretizations.
The SDC method is based on the Picard equation and a deferred correction procedure in the integral formulation,
driven by either the explicit or the implicit Euler marching scheme \cite{dutt2000spectral}.
This method is well-studied and has been further developed in both analysis and numerical aspects,
see e.g. \cite{ Minion2003SemiimplicitSD, layton2004conservative, huang2006accelerating, hagstrom2007spectral,  christlieb2010integral, minion2011hybrid}.
However, we believe that the use of SDC with low-rank methods is new.

For the method we propose in this paper, we start by computing a first order solution by mBUG,
and then perform successive updates by computing low-rank solutions to the Picard integral equation.
Rather than a straightforward application of SDC with mBUG, we consider two aspects to improve computational efficiency.
The first is   intermediate numerical rank control.
We perform detailed analysis and derive the optimal choice of truncation parameter depending on the correction levels based on error analysis.
It turns out that large truncation tolerances can be used in the initial and early correction levels,
which will reduce the numerical rank of the intermediate solutions.
The second aspect is a careful choice of subspaces in the successive correction to avoid inverting large linear systems (from the K- and L-steps in BUG).
For correction steps, we only need to invert a small linear system from the Galerkin step.
This is inspired by the insight in \cite{huang2006accelerating} to view SDC as a preconditioned Neumann series expansion for collocation schemes.
We prove that for Lipschitz continuous and bounded dynamical system, our proposed scheme is high order accurate in time.

Beyond accuracy, an important quality of the low-rank method is if the numerical rank accurately tracks the solution rank.
This is a property that is hard to validate by theory for the rank adaptive method (unless in special circumstances e.g. DLRA for a fixed rank problem).
We consider numerical rank control of our method by comparing two low-rank truncation strategies:
the hard truncation strategy by truncated singular value decomposition (tSVD) and the soft truncation strategy by soft thresholding.
Soft thresholding is defined based on a shrinkage operator applied to the singular values \cite{cai2010singular}.
It induces a non-expansive operator and has been used in iterative solvers for linear operator equations in low-rank format \cite{bachmayr2017iterative}.
We demonstrate numerically that compared to tSVD, soft thresholding offers better rank control for time-dependent problems in particular for higher order schemes for weakly (or non-)dissipative problems.
We conduct several numerical benchmark tests arising from Schr\"{o}dinger equation and advection-diffusion equations to validate the performance of our method in terms of accuracy and rank control.

The rest of the paper is organized as follows.
Section~\ref{sec:background-review} provides a review of the background of the mBUG scheme and SDC method.
In Section~\ref{sec:numerical-method}, we describe the proposed SDC-mBUG scheme and perform theoretical analysis.
Section~\ref{sec:numerical-results} offers numerical experiments, and Section~\ref{sec:conclusions-and-future-work} concludes the paper.

\section{Background review}\label{sec:background-review}

In this section, we review some basic concepts about low-rank algorithms and the SDC time discretization scheme.
In Section~\ref{sec:preliminary}, we gather the linear algebra notation used in this paper.
Section~\ref{sec:mBUG} reviews the mBUG method.
Section~\ref{sec:sdcm} introduces the high order SDC time discretization method based on the implicit Euler scheme.

\subsection{Notations and preliminaries}\label{sec:preliminary}

In this paper, we use the inner product of the matrix defined as follows, for $A, B \in \mathbb{R}^{m_1\times m_2}$,
\[
    \langle A,B\rangle := \sum_{i=1}^{m_1} \sum_{j=1}^{m_2} A_{ij} B_{ij}.
\]
$\|\cdot\|$ is used to represent the matrix Frobenius norm, which can be derived from the inner product.
\[
    \|A\| = \sqrt{\langle A,A\rangle} = \left(\sum_{i=1}^{m_1} \sum_{j=1}^{m_2} |A_{ij}|^2\right)^{1/2}.
\]

The rank-adaptive low-rank methods rely on a low-rank matrix truncation operator to ensure that the numerical solution maintains the low-rank property.
In this paper, we denote $\mathcal{T}_\varepsilon$ as a general low-rank truncation operator from $\mathbb{R}^{m_1\times m_2}$ to $\mathbb{R}^{m_1\times m_2}$,
which may refer to either of the two truncation operators defined below or any other operator satisfying the following two properties
\begin{align*}
    \rank(\mathcal{T}_\varepsilon(X)) \le \rank(X), \quad \text{and} \quad
    \|X -\mathcal{T}_\varepsilon(X)\| \le \varepsilon.
\end{align*}
It is important to clarify that although the truncation operators described here are mappings between matrices,
the output we obtain from the algorithm is the SVD decomposition of the truncated matrix.

We define $\mathcal{T}_\varepsilon(\cdot)$ based on a shrinkage operator applied to singular values. Specifically, we consider two types of thresholding as defined below.
For $x \in \mathbb{R}$, hard thresholding with parameter $\alpha>0$ is defined as
\begin{equation*}
    \mathcal{H}_{\alpha}(x) := (1 - \chi_{[-\alpha,\alpha]}(x))x
    = \begin{cases}
        x, & x > \alpha,               \\
        0, & -\alpha \le x \le \alpha, \\
        x, & x < -\alpha,
    \end{cases}
\end{equation*}
and soft thresholding with parameter $\alpha>0$ is defined as
\begin{equation*}
    \mathcal{S}_{\alpha}(x) := \sgn(x) \max\{|x|-\alpha,0\} =
    \begin{cases}
        x-\alpha, & x > \alpha,               \\
        0,        & -\alpha \le x \le \alpha, \\
        x+\alpha, & x < -\alpha.
    \end{cases}
\end{equation*}
Consider the reduced SVD of a matrix $X \in \mathbb{R}^{m_1 \times m_2}$ of rank $r$:
\[
    X = U S V^\top = U \diag(\sigma_1, \cdots, \sigma_{r}) V^\top,
\]
where $U$ and $V$ are, respectively, $m_1 \times r$ and $m_2 \times r$ matrices with orthonormal columns,
and the singular values $\sigma_i$ are positive.
The hard and soft thresholdings for matrices are defined as follows:
\begin{align*}
    \mathcal{H}_\alpha(X) & := U \diag(\mathcal{H}_\alpha(\sigma_1), \cdots, \mathcal{H}_\alpha(\sigma_{r})) V^\top, \\
    \mathcal{S}_\alpha(X) & := U \diag(\mathcal{S}_\alpha(\sigma_1), \cdots, \mathcal{S}_\alpha(\sigma_{r})) V^\top.
\end{align*}
The soft thresholding operator is the proximal operator associated with the nuclear norm.
Details about the soft thresholding operator can be found in \cite{hiriart1996convex}.
As $\alpha>0$ increases, the ranks of the matrices $\mathcal{H}_\alpha(X)$ and $\mathcal{S}_\alpha(X)$ gradually decrease to zero.

Our low-rank algorithm requires the optimal low-rank approximation matrix under the condition that the truncation error does not exceed a given tolerance $\varepsilon$.
By defining the auxiliary functions $\Delta^h(\beta)$ and $\Delta^s(\beta)$, which are non-decreasing,
\begin{align*}
    \Delta^h(\beta) & := \| X - \mathcal{H}_\beta(X) \|^2 = \sum_{\sigma_j \le \beta} \sigma_j^2,                                   \\
    \Delta^s(\beta) & := \| X - \mathcal{S}_\beta(X) \|^2 = \sum_{\sigma_j \le \beta} \sigma_j^2 + \sum_{\sigma_j > \beta} \beta^2.
\end{align*}
we can introduce two matrix truncation operators $\mathcal{T}^h_\varepsilon,\mathcal{T}^s_\varepsilon: \mathbb{R}^{m_1\times m_2} \to \mathbb{R}^{m_1\times m_2}$
\begin{align*}
    \mathcal{T}^h_\varepsilon(X) & := \mathcal{H}_{\alpha}(X),
    \quad \alpha = \argmax \{\beta >0 \mid \Delta^h(\beta) \le \varepsilon^2 \}, \\
    \mathcal{T}^s_\varepsilon(X) & := \mathcal{S}_{\alpha}(X),
    \quad \alpha = \argmax \{\beta >0 \mid \Delta^s(\beta) \le \varepsilon^2\}.
\end{align*}
Applied to each singular value,
hard truncation provides a natural approach to obtaining low-rank approximations by dropping entries of small absolute value.
Soft truncation not only replaces entries that have an absolute value below the threshold by zero,
but also decreases values of all remaining entries. In our numerical results, we will compare the two truncation schemes and demonstrate their different performances. In the rest of the paper, $\mathcal{T}_{\varepsilon}=\mathcal{T}^{h}_{\varepsilon}$ or $\mathcal{T}^{s}_{\varepsilon}$ depending on the algorithm used.

In the low-rank algorithm,   rounding (or summation of low rank matrices) is computed using Algorithm~\ref{algo:msum} denoted by $
    U S V^\top = \mathcal{T}^{\text{sum}}_\varepsilon(\sum_{j=1}^m U_j S_j V_j^\top).
 $ Note that 
in Algorithm~\ref{algo:msum}, $\mathcal{T}_\varepsilon$ can be $\mathcal{T}^{h}_\varepsilon$ or $\mathcal{T}^{s}_\varepsilon$ . We use the following notation.

We also use the notation $\mathcal{T}^{\text{sum}}_\varepsilon \circ F$
to denote the low-rank matrix summation operator $\mathcal{T}^{\text{sum}}_\varepsilon$ applied to the linear matrix function $F$
\[
    U_F S_F V_F^\top = \mathcal{T}^{\text{sum}}_\varepsilon \circ F (X,t)
    = \mathcal{T}^{\text{sum}}_\varepsilon(\sum_{j=1}^s A_j U S (B_j V)^\top + G(t)),
\]
where $X = U S V^\top$ and $G(t) = U_G S_G V_G^\top$ are the corresponding singular value decompositions.

\begin{algorithm}[htbp]
    \caption{Sum of low rank matrices}\label{algo:msum}
    \SetKwInput{KwPara}{Parameters}

    \KwIn{SVD of low rank matrices $X_j = U_j S_j V_j^\top, j=1\ldots m$}
    \KwOut{truncated SVD of their sum, $U S V^\top = \mathcal{T}^{\text{sum}}_\varepsilon(\sum_{j=1}^m U_j S_j V_j^\top)$}
    \KwPara{tolerance $\varepsilon$}

    Form $\widehat{U}=[U_1, \ldots, U_m], \, \widehat{S}=\diag(S_1, \ldots, S_m), \,\widehat{V}=[V_1, \ldots, V_m]$.

    Perform column pivoted QR factorization: $ [Q_1, R_1, P_1]=\qr(\widehat{U}),$ $[Q_2, R_2, P_2]=\qr(\widehat{V})$.

    Compute the low rank truncation:  $USV^\top = \mathcal{T}_{\varepsilon}(R_1 P_1^\top \,\widehat{S} \, P_2 R_2^\top)$ with $\mathcal{T}_{\varepsilon}=\mathcal{T}^{h}_{\varepsilon}$ or $\mathcal{T}^{s}_{\varepsilon}$.

    Form $U \leftarrow Q_1 U, V \leftarrow Q_2 V$.

    Return $U, S, V$.
\end{algorithm}

\subsection{The mBUG method}\label{sec:mBUG}
The mBUG method was proposed in \cite{appelo2024robust} as a correction to the rank adaptive BUG method \cite{ceruti2022rank}.
BUG method belongs to DLRA \cite{koch2007dynamical}, which was originally proposed as a fixed rank method, i.e. it solves the projected equation \eqref{eq:DLRAF}, where $\mathcal{P}_{X(t)}$ is the orthogonal projection onto the tangent space $T_{X(t)}\mathcal{M}_r$ of the rank $r$ matrix manifold $\mathcal{M}_r$
at $X(t)$.
If the SVD of $X \in \mathcal{M}_r$ is given by $U S V^\top$,
then the orthogonal projection $\mathcal{P}_X$ is \cite[Lemma 4.1]{koch2007dynamical}
\begin{equation}\label{eq:PUV}
    \mathcal{P}_{X}(Z) = U U^\top Z + Z V V^\top - U U^\top Z V V^\top, \quad \forall\, Z \in \mathbb{R}^{m_1\times m_2}.
\end{equation}
The tangent space of $\mathcal{M}_r$ at $X$ is
\[
    T_{X}\mathcal{M}_r=\left \{ [ U \quad U_\perp] \begin{bmatrix}
        \mathbb{R}^{r \times r}      & \mathbb{R}^{r \times (m_2-r)} \\
        \mathbb{R}^{(m_1-r)\times r} & 0^{(m_1-r) \times (m_2-r)}
    \end{bmatrix} [V \quad V_\perp]^\top\right \},
\]
where $U_\perp, V_\perp$ are orthogonal complements of $U, V$ in $\mathbb{R}^{m_1}, \mathbb{R}^{m_2},$ respectively.

The mBUG method was proposed in \cite{appelo2024robust} to fix the modeling error from BUG method, i.e. $\|(\mathcal{I}-\mathcal{P}_{X(t)})F(X,t)\|$. The mBUG method is an implicit method. It uses the BUG framework for obtaining the row and column subspaces and it
also includes the row and column spaces predicted from the explicit step truncation method.
The main steps of the rank adaptive mBUG method with implicit Euler discretization for equation \eqref{eqn:mode}
are highlighted in Algorithm~\ref{alg:rank-adaptive-mBUG}.
In this algorithm, by selecting the truncation tolerances $\varepsilon_f=\mathcal{O}(\Delta t)$, $\varepsilon_s = \mathcal{O}(\Delta t^2)$,
the method achieves first-order accuracy in time.

\begin{algorithm}[htbp]
    \caption{mBUG method \cite{appelo2024robust} 
    }
    \label{alg:rank-adaptive-mBUG}
    \SetKwInput{KwPara}{Parameters}

    \KwIn{numerical solution at time $t_n$:
        matrix $X_n \in \mathbb{R}^{m_1\times m_2}$ with rank $r_n$ in its SVD form $X_n = U_n S_n V_n^\top$.
    }
    \KwOut{numerical solution at time $t_{n+1}$:
        low-rank matrix $X_{n+1} \in \mathbb{R}^{m_1\times m_2}$ in its SVD form $X_{n+1} = U_{n+1} S_{n+1} V_{n+1}^\top$.
    }
    \KwPara{time step $\Delta t$, truncation tolerance $\varepsilon_f=\mathcal{O}(\Delta t)$, $\varepsilon_s = \mathcal{O}(\Delta t^2)$}

    {\bf} Compute the truncated sum of the right-hand side:
    $U_{F,n}\, S_{F,n}\, V_{F,n}^\top = \mathcal{T}^{\text{sum}}_{\varepsilon_f} \circ F(X_n,t_n)$.

        {\bf(K-step and L-step)}
    Solve two generalized Sylvester equations
    \begin{align*}
        K_{n+1} & = K_n + \Delta t F(K_{n+1}(V_n)^\top,t_{n+1}) V_n, \qquad K_n = U_n S_n,      \\
        L_{n+1} & = L_n + \Delta t F(U_n(L_{n+1})^\top,t_{n+1})^\top U_n, \qquad L_n = V_n S_n,
    \end{align*}
    to obtain $K_{n+1}$ and $L_{n+1}$, respectively.

        {\bf} Orthogonalize the merged spaces by column-pivoted QR factorization to get the prediction spaces
    \begin{align*}
        [\widehat{U},\sim,\sim] & = \qr([U_n,U_{F,n},K_{n+1}]), & [\widehat{V},\sim,\sim] & = \qr([V_n,V_{F,n},L_{n+1}]).
    \end{align*}

    {\bf(S-step)}
    Solve a generalized Sylvester equation to obtain $\widehat{S}_{n+1}$
    \[
        \widehat{S}_{n+1} = \widehat{S}_n + \Delta t \widehat{U}^\top F(\widehat{U}\widehat{S}_{n+1}\widehat{V}^\top,t_{n+1})  \widehat{V}, \qquad
        \widehat{S}_n = (\widehat{U}^\top U_n) S_n (\widehat{V}^\top V_n)^\top.
    \]

    {\bf}
    Compute the SVD matrix $P S_{n+1} Q^\top = \mathcal{T}_{\varepsilon_s}(\widehat{S}_{n+1})$.

    Form $U_{n+1} \leftarrow \widehat{U} P$, $V_{n+1} \leftarrow \widehat{V} Q$.

    Return $X_{n+1} = U_{n+1} S_{n+1} V_{n+1}^\top$.
\end{algorithm}

\subsection{Spectral deferred correction method}\label{sec:sdcm}

The mBUG scheme only has first-order accuracy in time. To improve the temporal accuracy, one approach is to directly embed it into high order temporal scheme, e.g. \cite{nakao2023reduced}.
In this work, we take an alternative approach, the   SDC method, which can achieve high order time accuracy by correcting the numerical solution obtained from a low-order scheme systematically.
This section provides a brief review of the main computational steps of the classical SDC scheme.
The derivation of these formulas can be found in \cite{dutt2000spectral}.
We solve the initial value problem
\begin{equation}
    \left\{
    \begin{aligned}
        \frac{d}{dt} X(t) & = F(X(t),t) , \,\,t \in [0,T], \\
        X(0)              & = X_0,
    \end{aligned}
    \right.\label{sdc-ode}
\end{equation}
where $X(t)$ is the solution, and $F$ is the given right-hand side function.
It is assumed that $F$ is sufficiently smooth so that the discussion of higher-order methods is appropriate.
The corresponding Picard integral equation to \eqref{sdc-ode} is
\begin{equation*}
    X(t) = X_0 + \int_0^t F(X(\tau),\tau)\,d\tau.
\end{equation*}
Suppose now the time interval $[0, T]$ is divided into $M$ non-overlapping intervals by the partition
\[
    0 = t_0 < t_1 < \dots < t_n < \dots < t_M = T.
\]
We will describe an implicit SDC method below, which is used to advance the solution from $t_n$ to $t_{n+1}$.
Let $\Delta t_{n} = t_{n+1} - t_n$ and $X_n$ denotes the numerical approximation of $X(t_n)$.
Divide the time interval $[t_n, t_{n+1}]$ into P subintervals by choosing the points $t_{n,m}$, $m = 0, 1, \dots ,P$,
such that
\[
    t_n = t_{n,0} < t_{n,1} <\dots < t_{n,m} < \dots <t_{n,P} = t_{n+1},
    \quad \Delta t_{n,m} := t_{n,m+1} - t_{n,m}.
\]
The points $\{t_{n,m}\}_{m=0}^P$ can be chosen to be the Legendre Gauss-Lobatto nodes on $[t_n, t_{n+1}]$ to avoid the instability of approximation at equispaced nodes for high order accuracy.

The main steps of the implicit SDC method from $t_n$ to $t_{n+1}$ are highlighted in Algorithm~\ref{alg:sdc}. The base scheme is the implicit Euler method.
We use $X^{(k)}_{n,m}$ to denote the approximation to $X(t_{n,m})$ obtained at the $k$-th level in the SDC method.
$I_{m}^{m+1} F(X^{(k)},t)$ is used to denote the integral  of the $P$-th degree interpolating polynomial on the $P+1$ points $(t_{n,m}, F(X^{(k)}_{n,m},t_{n,m}))_{m=0}^p$ over the subinterval $[t_{n,m},t_{n,m+1}]$,
which can be expressed in the following form,
\begin{equation} \label{eq:integral}
    \begin{aligned}
        I^{m+1}_m F(X^{(k)},t) & =
        \int_{t_{n,m}}^{t_{n,m+1}} \sum_{s=0}^P \frac{\prod_{j=0,j \neq s}^{m}(t-t_{n,j})}{\prod_{j=0,j \neq s}^{m}(t_{n,s}-t_{n,j})} F(X^{(k)}_{n,s},t_{n,s}) \,dt
        \\ & =
        \Delta t_{n,m}\sum_{s=0}^P \left(\frac{1}{ \Delta t_{n,m}}\int_{t_{n,m}}^{t_{n,m+1}} \frac{\prod_{j=0,j \neq s}^{m}(t-t_{n,j})}{\prod_{j=0,j \neq s}^{m}(t_{n,s}-t_{n,j})}\,dt\right)  F(X^{(k)}_{n,s},t_{n,s})
        \\ &=:
        \Delta t_{n,m} \sum_{s=0}^P w_{m,s} F(X^{(k)}_{n,s},t_{n,s})
    \end{aligned}
\end{equation}
where we define $\{w_{m,s}\}$ as the quadrature weights for the integral $[t_{n,m},t_{n,m+1}]$.

\begin{algorithm}[htbp]
    \caption{SDC method from $t_n$ to $t_{n+1}$}\label{alg:sdc}
    \SetKwInput{KwPara}{Parameters}

    \KwIn{initial data at $t_{n}$: $X_n$  }
    \KwOut{numerical solution at $t_{n+1}$: $X_{n+1}$}
    \KwPara{time step $\Delta t$, number of corrections $K$, number of steps $P$}

    \textbf{Stage 1: Initial Approximation} \\
    Set $X^{(1)}_{n,0} = X_n$.\\
    \For{$m = 0, \dots , P-1$}{\
    Use   implicit Euler scheme to compute approximate solution $\{X^{(1)}_{n,m}\}$ at the nodes $\{t_{n,m}\}_{m=0}^P$ recursively, i.e.
    \begin{equation*}
        X^{(1)}_{n,m+1} = X^{(1)}_{n,m} + \Delta t_{n,m} F(X^{(1)}_{n,m+1},t_{n,m+1}),\quad
        m = 0, \dots ,P-1.
    \end{equation*}
    }

    \textbf{Stage 2: Successive Corrections} \\
    \For{$k = 1, \dots , K$}{\
        Set $X^{(k+1)}_{n,0} = X_n$\\
        \For{$m = 0, \dots , P-1$}{\
            Solve the correction equation to obtain $\{X^{(k+1)}_{n,m}\}$,
            \begin{equation}
                \begin{aligned}
                    X^{(k+1)}_{n,m+1} = & X^{(k+1)}_{n,m} + \Delta t_{n,m} \left( F(X^{(k+1)}_{n,m+1},t_{n,m+1}) - F(X^{(k)}_{n,m+1},t_{n,m+1}) \right) \\
                                        & + I^{m+1}_m F(X^{(k)},t).
                \end{aligned}
            \end{equation}}
    }
    Return $X_{n+1}=X^{(K+1)}_{n,P}$.
\end{algorithm}

It can be proven that the SDC scheme has $\min(K+1, P+1)$-th order time accuracy. The local truncation error is analyzed in \cite{dutt2000spectral}.

\begin{theorem}
    Suppose that $X_n = X(t_n)$, the local truncation error of the SDC scheme is
    \[
        \| X_{n+1} - X(t_{n+1}) \| = \mathcal{O}\left( h^{\min{(K+2,P+2)}} \right),
    \]
    where $h = \max_{m}\{\Delta t_{n,m}\}$.
\end{theorem}

\begin{remark*}
    Since the order of accuracy is   $\min{(K+1, P+1)}$, usually we take $K = P.$
\end{remark*}

\section{Numerical method}\label{sec:numerical-method}

We now present the proposed SDC-mBUG method for \eqref{eqn:mode}.
In Section~\ref{sec:sdc-mBUG}, we describe the algorithm.
In Section~\ref{sec:theoretical-analysis}, we prove the high order accuracy of the SDC-mBUG scheme.

\subsection{The SDC-mBUG method}\label{sec:sdc-mBUG}

We use the same notations as Section~\ref{sec:sdcm}, namely, $X^{(k)}_{n,m}$ is the approximation to $X(t_{n,m})$ at the $k$-th level in the SDC method.
To advance from $t_n$ to $t_{n+1}$, our algorithm consists of two stages: in the first stage, an initial approximation is obtained using the first-order mBUG scheme,
and in the second stage, successive corrections are computed based on the SDC scheme, where the main idea of the mBUG scheme is also applied.
Here, we use $X^{(k)}_{n,m} = U^{(k)}_{n,m} \, S^{(k)}_{n,m} \, (V^{(k)}_{n,m})^\top$ to represent the SVD of $X^{(k)}_{n,m}$. \\

{\bf Compute the initial approximation:} We use the mBUG method to compute the initial approximation: For $m = 0, \dots, P-1$,
\[
    X^{(1)}_{n,m+1} = X^{(1)}_{n,m} + \Delta t_{n,m} F(X^{(1)}_{n,m+1},t_{n,m+1}),
\]
with $X^{(1)}_{n,0} = X_n$. Specifically,
\begin{enumerate}[label=\textbf{Step I-\arabic*}]
    \item Compute $U_{F,n,m}^{(1)}\, S_{F,n,m}^{(1)}\, (V_{F,n,m}^{(1)})^\top := \mathcal{T}^{\text{sum}}_{\varepsilon_f} \circ F(X^{(1)}_{n,m},t_{n,m})$.

    \item K-step and L-step for $k=1$: Solve two generalized Sylvester equations to obtain $K^{(1)}_{n,m+1}$ and $L^{(1)}_{n,m+1}$, respectively,
          \begin{equation}\label{eq:sdc-mBUG-i1}
              \begin{aligned}
                  K^{(1)}_{n,m+1} & = X^{(1)}_{n,m} V^{(1)}_{n,m} + \Delta t_{n,m} F(K^{(1)}_{n,m+1} (V^{(1)}_{n,m})^\top,t_{n,m+1}) V^{(1)}_{n,m},              \\
                  L^{(1)}_{n,m+1} & = (X^{(1)}_{n,m})^\top U^{(1)}_{n,m} + \Delta t_{n,m} F(U^{(1)}_{n,m} (L^{(1)}_{n,m+1})^\top,t_{n,m+1})^\top U^{(1)}_{n,m}.
              \end{aligned}
          \end{equation}

    \item Update spaces for S-step:
          \begin{align*}
              [\widehat{U},\sim,\sim] & := \qr([U^{(1)}_{n,m},U_{F,n,m}^{(1)},K^{(1)}_{n,m+1}]), &
              [\widehat{V},\sim,\sim] & := \qr([V^{(1)}_{n,m},V_{F,n,m}^{(1)},L^{(1)}_{n,m+1}]).
          \end{align*}

    \item S-step for $k=1$: Solve the following equation to obtain $\widehat{S}^{(1)}_{n,m+1}$,
          \begin{equation}\label{eq:sdc-mBUG-i2}
              \begin{aligned}
                  \widehat{S}^{(1)}_{n,m+1} ={}
                   & \widehat{U}^\top X^{(1)}_{n,m}\widehat{V}
                  + \Delta t_{n,m} \widehat{U}^\top F(\widehat{U}\widehat{S}^{(1)}_{n,m+1}\widehat{V}^\top,t_{n,m+1}) \widehat{V}.
              \end{aligned}
          \end{equation}

    \item Compute $P S^{(1)}_{n,m+1} Q^\top := \mathcal{T}_{\varepsilon_{s}}(\widehat{S}^{(1)}_{n,m+1})$.

    \item Form $U^{(1)}_{n,m+1} = \widehat{U} P$, and
          $V^{(1)}_{n,m+1} = \widehat{V} Q$.
\end{enumerate}

{\bf Compute successive corrections:} 
For $k=1, \dots, K$, we would like to obtain $X^{(k+1)}_{n,m+1}$ recursively for $m = 0, \dots, P-1,$ which are low rank approximations to $\widetilde{X}^{(k+1)}_{n,m+1},$ that satisfies
\begin{align}\label{eq:csr}
    \widetilde{X}^{(k+1)}_{n,m+1} ={} & X^{(k+1)}_{n,m} +
    \Delta t_{n,m} \big\{F(\widetilde{X}^{(k+1)}_{n,m+1},t_{n,m+1}) - F(X^{(k)}_{n,m+1},t_{n,m+1})\big\} + I^{m+1}_m F(X^{(k)},t),
\end{align}
with initial value $X^{(k+1)}_{n,0} = X_n$.
Here, $I_{m}^{m+1} F(X^{(k)},t)$ denotes the integral  of the $P$-th degree interpolating polynomial on the $P+1$ points $(t_{n,m}, F(X^{(k)}_{n,m},t_{n,m}))_{m=0}^p$ over the subinterval $[t_{n,m},t_{n,m+1}]$.

In practice, \eqref{eq:csr} is computed by the mBUG method without the $K$ and $L$-steps. 
Namely, for $k = 1, \dots, K$: set $X^{(k+1)}_{n,0} = X_n$; then,
\begin{enumerate}[label=\textbf{Step C-\arabic*}]
    \item \label{step-c1} Compute $U_{F,n,m}^{(k)} S_{F,n,m}^{(k)} (V_{F,n,m}^{(k)})^\top := \mathcal{T}^{\text{sum}}_{\varepsilon^{(k)}_{f}} \circ F(X^{(k)}_{n,m},t_{n,m})$ for $m = 0, \dots, P$.

    \item \label{step-c2} Compute the truncated sum in level $k$ for $m = 0, \dots, P-1$
          \begin{align*}
                   & U_{R,n,m}^{(k)} S_{R,n,m}^{(k)} (V_{R,n,m}^{(k)})^\top                                                                              \\
              :={} & \mathcal{T}^{\text{sum}}_{\varepsilon^{(k)}_{r}}\left(- \Delta t_{n,m} U_{F,n,m+1}^{(k)} S_{F,n,m+1}^{(k)} (V_{F,n,m+1}^{(k)})^\top
              + \Delta t_{n,m} \sum_{s=0}^P w_{m,s} U_{F,n,s}^{(k)} S_{F,n,s}^{(k)} (V_{F,n,s}^{(k)})^\top \right).
          \end{align*}

    \item \label{step-c3} For $m=0, \ldots, P-1$:
          \begin{itemize}
              \item Update spaces for S-step:
              \begin{equation}\label{eq:S-step}
                \begin{aligned}
                    [\widehat{U},\sim,\sim] & := \qr([U^{(k+1)}_{n,m}, U^{(k)}_{F,n,m+1}, U^{(k)}_{R,n,m}]),  \\
                    [\widehat{V},\sim,\sim] & := \qr([V^{(k+1)}_{n,m}, V^{(k)}_{F,n,m+1}, V^{(k)}_{R,n,m}]).
                \end{aligned}
            \end{equation}
            The subspace choices are largely motivated by the mBUG method and the proof of Theorem \ref{prop:sdc-mBUG}.

              \item \label{step-c4} Solve the following equation to obtain $\widehat{S}^{(k+1)}_{n,m+1}$,
                    \begin{equation}\label{eq:sdc-mBUG-m2}
                        \begin{aligned}
                            \widehat{S}^{(k+1)}_{n,m+1} ={} & \widehat{U}^\top X^{(k+1)}_{n,m} \widehat{V}
                            + \Delta t_{n,m} \widehat{U}^\top F(\widehat{U}\widehat{S}^{(k+1)}_{n,m+1}\widehat{V}^\top,t_{n,m+1}) \widehat{V}        \\
                            {}                              & + \widehat{U}^\top U_{R,n,m}^{(k)} S_{R,n,m}^{(k)} (V_{R,n,m}^{(k)})^\top \widehat{V}.
                        \end{aligned}
                    \end{equation}

              \item \label{step-c5} Compute $P S^{(k+1)}_{n,m+1} Q^\top := \mathcal{T}_{\varepsilon^{(k+1)}_{s}}(\widehat{S}^{(k+1)}_{n,m+1})$.

              \item Form $U^{(k+1)}_{n,m+1} = \widehat{U} P$, $V^{(k+1)}_{n,m+1} = \widehat{V} Q$.
          \end{itemize}
\end{enumerate}

Finally, we take $X_{n+1} = X^{(K+1)}_{n,P} = U^{(K+1)}_{n,P} S^{(K+1)}_{n,P} (V^{(K+1)}_{n,P})^\top$ as the low-rank numerical solution at time $t_{n+1}$.

\vspace{0.3cm}

In \eqref{eq:S-step}, the selection of $\widehat{U}$ and $\widehat{V}$ follows a similar design principle to the mBUG method, i.e., using an explicit formulation for prediction.
Compared to the BUG and mBUG methods, we omit the $K$- and $L$-steps in the second stage, which are costly linear solves.
This is inspired by the insight in \cite{huang2006accelerating} to view SDC as a preconditioned Neumann series expansion for collocation schemes.

Our algorithm involves many low-rank truncation operations.
The error tolerance of low-rank truncation operators will affect the accuracy of the method.
The specific strategy for selecting these truncation thresholds is as follows: set $P=K$,
\[
    \begin{aligned}
         & \varepsilon_{s} = \mathcal{O}(h^2), \quad \varepsilon_{f} = \mathcal{O}(h), \quad
        \varepsilon^{(k+1)}_{s} = \varepsilon^{(k)}_{r} = \mathcal{O}(h^{k+2}), \quad
        \varepsilon^{(k)}_{f} = \mathcal{O}(h^{k+1}), \quad k=1,\dots,K.
    \end{aligned}
\]
As we will show Section \ref{sec:theoretical-analysis}, this strategy is sufficient to achieve $(K+1)$-th order accuracy. Moreover, as we can see, in the early stages (smaller $k$) we are allowed to take larger truncation tolerances, which can effectively control the rank and numerical cost. With given tolerance parameters, in our numerical experiments, we test both the hard truncation and the soft truncation.
The results of these experiments demonstrate that while both truncation achieve optimial accuracy, they differ in   performance in rank control.

\subsection{Theoretical analysis}\label{sec:theoretical-analysis}

We make standard assumptions as in the error analysis of BUG and mBUG schemes, i.e.   the operator $F$ is Lipschitz continuous in both variables and bounded, i.e., there exist constants $L, B >0$ such that
\begin{equation}\label{eq:lipschitz}
    \| F(Y,s) - F(Z,t) \| \le L \| Y - Z\| + L \vert s-z \vert,\quad
    \| F(X,t) \| \le B.
\end{equation}
Below, first we review the local truncation error of the first order mBUG scheme in \cite{appelo2024robust}. Then, we present the local truncation error of our proposed SDC-mBUG scheme.

\begin{lemma}[Local truncation error of first order mBUG scheme \cite{appelo2024robust}] \label{lemma:first-order-mBUG}
    Suppose $F$ satisfies conditions \eqref{eq:lipschitz}.
    If we denote $X(t_{n+1})$ as the exact solution
    and $X_{n+1}$ the numerical solution obtained from the first order mBUG scheme in Algorithm \ref{alg:rank-adaptive-mBUG} from $t_n$ to $t_{n+1} = t_n + h$ with $X_n=X(t_n)$, then for sufficiently small $h$ (i.e. $h L \le \gamma < 1$ for some $\gamma > 0$), the local truncation error is bounded by
    \begin{equation*}
        \| X_{n+1}-X(t_{n+1})\| \le C (h^2 + h \varepsilon_f + \varepsilon_s),
    \end{equation*}
    where $C$ is a constant that only depends on $L$, $B$ and $\gamma$.
\end{lemma}

\begin{theorem}[Local truncation error of SDC-mBUG scheme] \label{prop:sdc-mBUG}
    Under the assumptions \eqref{eq:lipschitz},
    the local truncation error obtained with the SDC-mBUG scheme satisfies
    \begin{align} \label{eq:truncation_error}
        \| X_{n+1} - X(t_{n+1}) \| \le  C_{K+1} \left( \delta^{(K+1)} +h^{\min{(K+2,P+2)}} \right),
    \end{align}
    where the time step $h = t_{n+1} - t_n$ is small enough (i.e. $h L \le \gamma < 1$ for some $\gamma > 0$),
    $C_{K+1}$ is a constant that depends on $L,B,P,k,\gamma$,
    and $\delta^{(k+1)}$ 
    satisfies
    \begin{align*}
        \delta^{(1)}   & = \mathcal{O}(\varepsilon_s + \varepsilon_f h),                                                                                                                                                                \\
        \delta^{(k+1)} & = \mathcal{O} \left(h \delta^{(k)} + \left(\varepsilon^{(k+1)}_{s} + h \varepsilon^{(k)}_{f} + \varepsilon^{(k)}_{r}\right)\right)                                                                             \\
                       & = \mathcal{O}\left(h^{k} \varepsilon_s + h^{k+1} \varepsilon_f + \sum_{i=1}^{k}  h^{k-i}\left(  \varepsilon^{(i+1)}_{s} + h \varepsilon^{(i)}_{f} + \varepsilon^{(i)}_{r} \right)\right), \quad k = 1,\dots,K.
    \end{align*}
\end{theorem}

\begin{proof}
    We will use   induction   to prove that
    \begin{align}\label{eq:truncation}
        \| X(t_{n,m}) - X^{(k)}_{n,m} \| \le  C_{k} \left( \delta^{(k)} + h^{\min(k+1,P+2)} \right), \quad m=0,\dots,P, \,\, k=1, \ldots, K+1.
    \end{align}
    It is obvious that $k=1$ meets this requirement: the local truncation error of the first-order mBUG scheme is
    $C (h^2 + h \varepsilon_{f} + \varepsilon_{s}) = C_1 (\delta^{(1)} +  h^2)$ according to Lemma~\ref{lemma:first-order-mBUG}.
    Assume that \eqref{eq:truncation} is true for level $k$, for all $m=0,\dots, P$,
    in the following, we will prove that \eqref{eq:truncation} also holds for $k+1$.
    In particular, we will estimate $\| X(t_{n,m+1}) -\widetilde{X}^{(k+1)}_{n,m+1} \|$ first, and then estimate $\| X^{(k+1)}_{n,m+1} - \widetilde{X}^{(k+1)}_{n,m+1} \|$.
    Here, $\widetilde{X}^{(k+1)}_{n,m+1}$ is  defined in \eqref{eq:csr}.

    Taking the difference between \eqref{eq:csr} and
    \begin{align*}
        X(t_{n,m+1}) = X(t_{n,m}) + \int_{t_{n,m}}^{t_{n,m+1}} F(X(t),t) \,dt,
    \end{align*}
    we have that
    \begin{align*}
        X(t_{n,m+1}) - \widetilde{X}^{(k+1)}_{n,m+1}
        ={} & X(t_{n,m}) - X^{(k+1)}_{n,m}
        \\ &+ \underbrace{\int_{t_{n,m}}^{t_{n,m+1}} F(X(t),t) \,dt - I^{m+1}_m F(X(t),t)}_{(A)}
        \\ &+ \underbrace{I^{m+1}_m F(X(t),t) - I^{m+1}_m F(X^{(k)},t)}_{(B)}
        \\ & \underbrace{-\Delta t_{n,m} \big\{ F(\widetilde{X}^{(k+1)}_{n,m+1} ,t_{n,m+1}) - F(X^{(k)}_{n,m+1},t_{n,m+1})\big\}}_{(C)}.
    \end{align*}

    \noindent
    Since $I^{m+1}_m F(X(t),t)$ is the integral of the $P$-th degree interpolating polynomial on the $P+1$ points
    $\{(t_{n,m}, F(X(t_{n,m})))\}_{m=0}^P$ over the subinterval $[t_{n,m}, t_{n,m+1}]$, which is accurate to the order $\mathcal{O}(h^{P+2})$,
    we have
    \begin{align*}
        \| (A) \| ={} & \| \int_{t_{n,m}}^{t_{n,m+1}} F(X(t),t) \,dt - I^{m+1}_m F(X(t),t) \| \le C_A h^{P+2},
    \end{align*}
    and
    \begin{align*}
        \| (B) \| = \| I^{m+1}_m F(X(t),t) - I^{m+1}_m F(X^{(k)},t) \| \le C_B C_k h  (\delta^{(k)} + h^{\min(k+1,P+2)}),
    \end{align*}
    where $C_A$ and $C_B$ are independent of $h,$ but may depend on $L, B$. On the other hand,
    \begin{align*}
        \|(C)\| ={} & \| - \Delta t_{n,m} \big\{F(\widetilde{X}^{(k+1)}_{n,m+1},t_{n,m+1})  - F(X^{(k)}_{n,m+1},t_{n,m+1})\big\} \|
        \\
        \leq {}     & \Delta t_{n,m} L \| \widetilde{X}^{(k+1)}_{n,m+1}  - X^{(k)}_{n,m+1} \|
        \\
        \leq {}     & h L \| X(t_{n,m+1}) - \widetilde{X}^{(k+1)}_{n,m+1} \|
        + h L \| X(t_{n,m+1}) - X^{(k)}_{n,m+1} \|                                                                                  \\
        \leq {}     & h L \| X(t_{n,m+1}) - \widetilde{X}^{(k+1)}_{n,m+1} \|
        + h L C_k \left( \delta^{(k)} +  h^{\min(k+1,P+2)} \right).
    \end{align*}
    Then combining these three parts, we have that
    \begin{equation}\label{eq:star}
        \begin{aligned}
            {}    & (1-h L) \| X(t_{n,m+1}) - \widetilde{X}^{(k+1)}_{n,m+1} \| \\
            \le{} &
            \| X(t_{n,m}) - X^{(k+1)}_{n,m}\|
            + C_A h^{P+2} + h(C_B+ L) C_k(\delta^{(k)} + h^{\min(k+1,P+2)}).
        \end{aligned}
    \end{equation}

    In the following, we will show that the following inequality holds
    \begin{align}\label{eq:claim}
        (1 - h L) \| \widetilde{X}^{(k+1)}_{n,m+1} - X^{(k+1)}_{n,m+1} \|
        \le{} &
        2 h L \|  \widetilde{X}^{(k+1)}_{n,m+1} - X^{(k)}_{n,m+1}\|
        + (1-h L) \varepsilon^{(k+1)}_{s} + h \varepsilon^{(k)}_{f} + \varepsilon^{(k)}_{r}.
    \end{align}
    Define $\widehat{X}^{(k+1)}_{n,m+1} = \widehat{U} \widehat{S}^{(k+1)}_{n,m+1} \widehat{V}^\top$. Based on the definition of $\widehat{S}^{(k+1)}_{n,m+1}$ in the S-step \eqref{eq:sdc-mBUG-m2}, we derive that
    \begin{equation*}
        \begin{aligned}
            \widehat{X}^{(k+1)}_{n,m+1}
            ={} & \widehat{U} \widehat{S}^{(k+1)}_{n,m+1} \widehat{V}^\top                  \\
            ={} & \widehat{U} \widehat{U}^\top X^{(k+1)}_{n,m} \widehat{V} \widehat{V}^\top
            + \Delta t_{n,m} \widehat{U} \widehat{U}^\top F(\widehat{X}^{(k+1)}_{n,m+1},t_{n,m+1}) \widehat{V} \widehat{V}^\top
            + \widehat{U} \widehat{U}^\top R^{(k)}_{n,m} \widehat{V} \widehat{V}^\top
        \end{aligned}
    \end{equation*}
    Here, $R_{n,m}^{(k)} = U_{R,n,m}^{(k)} S_{R,n,m}^{(k)} (V_{R,n,m}^{(k)})^\top$.
    Since the column space of $X^{(k+1)}_{n,m}$ is a subset of the column space of $\widehat{U}$,
    we have $\widehat{U} \widehat{U}^\top X^{(k+1)}_{n,m} = X^{(k+1)}_{n,m}$.
    Similarly, $ X^{(k+1)}_{n,m} \widehat{V} \widehat{V}^\top = X^{(k+1)}_{n,m}$.
    Therefore, $\widehat{U} \widehat{U}^\top X^{(k+1)}_{n,m} \widehat{V} \widehat{V}^\top = X^{(k+1)}_{n,m}$.
    By a similar argument, we also have
    \begin{align*}
        \widehat{U} \widehat{U}^\top F({X}^{(k+1)}_{n,m+1},t_{n,m+1}) \widehat{V} \widehat{V}^\top = F({X}^{(k+1)}_{n,m+1},t_{n,m+1}), \qquad
        \widehat{U} \widehat{U}^\top R^{(k)}_{n,m} \widehat{V} \widehat{V}^\top = R^{(k)}_{n,m}.
    \end{align*}
    This gives
    \begin{equation*}
        \begin{aligned}
            \widehat{X}^{(k+1)}_{n,m+1} ={} & X^{(k+1)}_{n,m}
            + \Delta t_{n,m} F(X^{(k)}_{n,m+1},t_{n,m+1})  + R^{(k)}_{n,m}                                                                                                                          \\
                                            & + \Delta t_{n,m} \widehat{U} \widehat{U}^\top (F(\widehat{X}^{(k+1)}_{n,m+1},t_{n,m+1}) - F(X^{(k)}_{n,m+1},t_{n,m+1})) \widehat{V} \widehat{V}^\top.
        \end{aligned}
    \end{equation*}
    From Step C-1 and Step C-2, we know that the truncation errors are introduced
    \begin{align*}
        \|\Delta t_{n,m} \big\{F(\widetilde{X}^{(k+1)}_{n,m+1},t_{n,m+1}) - F(X^{(k)}_{n,m+1},t_{n,m+1})\big\} + I^{m+1}_m F(X^{(k)},t) - R^{(k)}_{n,m}\| \le h \varepsilon^{(k)}_{f} + \varepsilon^{(k)}_{r}
    \end{align*}
    Therefore,
    \begin{align*}
        \|\widetilde{X}^{(k+1)}_{n,m+1} - \widehat{X}^{(k+1)}_{n,m+1}\|
        \leq {} & \| \Delta t_{n,m}  \left( F(\widetilde{X}^{(k+1)}_{n,m+1},t_{n,m+1})
        - F(X^{(k)}_{n,m+1},t_{n,m+1})
        \right)\|
        \\ & + \| \Delta t_{n,m} \widehat{U} \widehat{U}^\top (F(\widehat{X}^{(k+1)}_{n,m+1},t_{n,m+1}) - F(X^{(k)}_{n,m+1},t_{n,m+1})) \widehat{V} \widehat{V}^\top \|\\
                & + h \varepsilon^{(k)}_{f} + \varepsilon^{(k)}_{r}.
    \end{align*}
    In fact, for any matrix $Z$,
    \begin{align*}
        \| \widehat{U} \widehat{U}^\top Z \widehat{V} \widehat{V}^\top \|
        \le \| \widehat{U} \widehat{U}^\top \|_2 \,\,\| Z \| \,\, \| \widehat{V} \widehat{V}^\top \|_2
        \le \| Z \|,
    \end{align*}
    where $\|\cdot \|_2$ is the matrix 2-norm.
    Thus
    \begin{align*}
        \| \widetilde{X}^{(k+1)}_{n,m+1} - \widehat{X}^{(k+1)}_{n,m+1}\|
        \le{} & h  \|  F(\widetilde{X}^{(k+1)}_{n,m+1},t_{n,m+1})
        - F(X^{(k)}_{n,m+1},t_{n,m+1})
        \|                                                                                                                                          \\
              & + h \| F(\widehat{X}^{(k+1)}_{n,m+1},t_{n,m+1}) - F(X^{(k)}_{n,m+1},t_{n,m+1}) \| + h \varepsilon^{(k)}_{f} + \varepsilon^{(k)}_{r}
        \\
        \le{} & h L  \|  \widetilde{X}^{(k+1)}_{n,m+1}
        - X^{(k)}_{n,m+1}
        \|  + h L \| \widehat{X}^{(k+1)}_{n,m+1} - X^{(k)}_{n,m+1} \|+ h \varepsilon^{(k)}_{f} + \varepsilon^{(k)}_{r}
        \\
        \le{} & 2 h L  \|  \widetilde{X}^{(k+1)}_{n,m+1} - X^{(k)}_{n,m+1}\|
        + h L \| \widetilde{X}^{(k+1)}_{n,m+1} - \widehat{X}^{(k+1)}_{n,m+1} \|+ h \varepsilon^{(k)}_{f} + \varepsilon^{(k)}_{r}.
    \end{align*}
    This tells
    \begin{align*}
        (1- h L) \| \widetilde{X}^{(k+1)}_{n,m+1} - \widehat{X}^{(k+1)}_{n,m+1} \|
        \le 2 h L
        \|  \widetilde{X}^{(k+1)}_{n,m+1} - X^{(k)}_{n,m+1}\|
        + h \varepsilon^{(k)}_{f} + \varepsilon^{(k)}_{r}.
    \end{align*}
    Since we apply the low rank truncation $\mathcal{T}_{\varepsilon^{(k+1)}_{s}}$ on $\widehat{X}^{(k+1)}_{n,m+1}$,
    \begin{align*}
        \| X^{(k+1)}_{n,m+1} - \widehat{X}^{(k+1)}_{n,m+1} \| = \|  \mathcal{T}_{\varepsilon^{(k+1)}_{s}}(\widehat{X}^{(k+1)}_{n,m+1}) -  \widehat{X}^{(k+1)}_{n,m+1}  \|
        \le \varepsilon^{(k+1)}_{s},
    \end{align*}
    we have that
    \begin{align*}
        (1 - h L) \| \widetilde{X}^{(k+1)}_{n,m+1} - X^{(k+1)}_{n,m+1} \|
        \le{} &
        2 h L \|  \widetilde{X}^{(k+1)}_{n,m+1} - X^{(k)}_{n,m+1}\|
        + (1-h L) \varepsilon^{(k+1)}_{s} + h \varepsilon^{(k)}_{f} + \varepsilon^{(k)}_{r}
    \end{align*}
    Therefore, the claim \eqref{eq:claim} holds.

    Using \eqref{eq:star}, \eqref{eq:claim}, and the induction hypothesis, we can finally complete the proof of accuracy:
    \begin{align*}
        {}     & \| X^{(k+1)}_{n,m+1} - X(t_{n,m+1}) \|                                                                                                              \\
        \le{}  & \| X^{(k+1)}_{n,m+1} - \widetilde{X}^{(k+1)}_{n,m+1} \| + \| \widetilde{X}^{(k+1)}_{n,m+1} - X(t_{n,m+1}) \|
        \\
        \le{}  & \frac{2hL}{1-hL}
        \|  \widetilde{X}^{(k+1)}_{n,m+1} - X^{(k)}_{n,m+1}\|
        + \varepsilon^{(k+1)}_{s} + \frac{1}{1-h L}(h \varepsilon^{(k)}_{f} + \varepsilon^{(k)}_{r})
        + \| \widetilde{X}^{(k+1)}_{n,m+1} - X(t_{n,m+1}) \|                                                                                                         \\
        \le {} & \frac{2hL}{1-hL} \| X(t_{n,m+1}) - X^{(k)}_{n,m+1} \| + \varepsilon^{(k+1)}_{s} + \frac{1}{1-h L}(h \varepsilon^{(k)}_{f} + \varepsilon^{(k)}_{r})
        + \frac{1+hL}{1-hL} \|  \widetilde{X}^{(k+1)}_{n,m+1} - X(t_{n,m+1})\|
        \\
        \le{}  &
        \frac{2hL}{1-hL} C_k (\delta^{(k)} +  h^{\min(k+1,P+2)}) + \varepsilon^{(k+1)}_{s} + \frac{1}{1-h L}(h \varepsilon^{(k)}_{f} + \varepsilon^{(k)}_{r})        \\
               & + \frac{(1+h L)}{(1-h L)^2}
        \left(\|X(t_{n,m}) - X^{(k+1)}_{n,m}\| + C_A h^{P+2} + h(C_B+ L)C_k (\delta^{(k)} + h^{\min(k+1,P+2)})\right)
        \\
        \le{}  &
        \frac{2hL}{1-\gamma} C_k (\delta^{(k)} +  h^{\min(k+1,P+2)}) + \varepsilon^{(k+1)}_{s} + \frac{1}{1-\gamma}(h \varepsilon^{(k)}_{f} + \varepsilon^{(k)}_{r}) \\
               & + \frac{(1+\gamma)}{(1-\gamma)^2}
        \left(\|X(t_{n,m}) - X^{(k+1)}_{n,m}\| + C_A h^{P+2} + h(C_B+L)C_k(\delta^{(k)} + h^{\min(k+1,P+2)})\right)
    \end{align*}
    By induction on $m$, it follows that \eqref{eq:truncation} holds at the
    $(k+1)$-th level: there exists a constant $C_{k+1} > 0$ such that
    \begin{align*}
        \| X(t_{n,m}) - X^{(k+1)}_{n,m} \| \le C_{k+1}(\delta^{(k+1)} +  h^{\min(k+2,P+2)}),
    \end{align*}
    where $\delta^{(k+1)}$   satisfies
    \[
        \delta^{(k+1)} = \mathcal{O} \left(h \delta^{(k)} + \left(\varepsilon^{(k+1)}_{s} + h \varepsilon^{(k)}_{f} + \varepsilon^{(k)}_{r}\right)\right).
    \]
    By setting $k = K$,$m=P$ in \eqref{eq:truncation}, we can get
    \begin{align*}
        \| X_{n+1} - X(t_{n+1}) \| = \| X^{(K+1)}_{n,P+1} - X(t_{n,P+1}) \| \le C_{k+1} (\delta^{(K+1)} + (h^{\min(K+2,P+2)})).
    \end{align*}
    The proof is complete.
\end{proof}

\begin{remark}
    By selecting $P=K$, we can derive
    \[
        \begin{aligned}
             & \varepsilon_{s} = \mathcal{O}(h^2), \quad \varepsilon_{f} = \mathcal{O}(h), \quad
            \varepsilon^{(k+1)}_{s} = \varepsilon^{(k)}_{r} = \mathcal{O}(h^{k+2}), \quad \varepsilon^{(k)}_{f} = \mathcal{O}(h^{k+1}), \quad k=1,\dots,K
        \end{aligned}
    \]
    is needed to obtain $(K+1)$-th order accuracy
    \[
        \| X_{n+1} - X(t_{n+1}) \| = \mathcal{O}(h^{K+2}).
    \]
\end{remark}

\section{Numerical results}\label{sec:numerical-results}

In this section, we present numerical results that illustrate the performance of the SDC-mBUG method.
In all numerical tests, we impose periodic boundary conditions on the solution domain, $\Omega=[-2\pi,2\pi]^2.$
We use Fourier collocation methods \cite{trefethen2000spectral}  to discretize the spatial variables.
The mesh size is $(N_x, N_y) = (200,200)$ to ensure temporal error dominates.
We compare the numerical solutions obtained by the SDC-mBUG method with the (full-rank) reference solutions obtained via the classical fourth-order Runge-Kutta scheme on a refined mesh, specifically
the time evolution of the $L^2$ error and the numerical rank.
Hard and soft truncation are compared.
The truncation coefficients used in the SDC-mBUG scheme for these numerical examples are as follows,
\[
    \begin{aligned}
        \varepsilon_{s} = C h^2, \quad \varepsilon_{f} = C h, \quad
        \varepsilon^{(k)}_{f} = C h^{k+1},\quad \varepsilon^{(k+1)}_{s} = \varepsilon^{(k)}_{r} = C h^{k+2}, \quad k=1,\dots,K,
    \end{aligned}
\]
where
\[
    C = 2 \left( \frac{4\pi}{N_x} + \frac{4\pi}{N_y} \right)^{-1}.
\]
We use the following notations: SDC-mBUG-X-H/S denotes the X-order SDC-mBUG scheme with hard/soft truncation, respectively.

To examine the rank evolution of the reference solution,
we perform low-rank truncation with the same  thresholds on the exact solution or full-rank reference solutions to obtain the corresponding ranks, denoted as Ref-X-H/S, where
Ref-3-H/S represents that the hard/soft truncation with a tolerance of $C h^4$ applied to the full-rank reference solutions, and Ref-4-H/S represents  hard/soft truncation with a tolerance of $C h^5$ applied to the reference solutions.





\begin{example}[Manufactured solution]\label{newEg1}
    We consider the following equation
    \[
        u_t - y u_x + x u_y = d(u_{xx} + u_{yy}) + \varphi,
    \]
    where the source term
    \[
        \varphi(x,y,t) = (6d - 4xy - 4d(x^2 + 9y^2)) \exp(-(x^2 + 3y^2 + 2dt)),
    \]
    and the diffusion coefficient $d = 1/5$.
    We start from the rank-1 initial data $u(x,y,0) = \exp(-(x^2 + 3y^2))$,
    The exact solution is $u(x,y,t) = \exp(-(x^2 + 3y^2 + 2d\,t))$, which is of rank 1 at all times. 

    We evolve the solution until time $T = \pi$ with a time step $\Delta t = \pi / N_t$.
    The errors and rates of convergence can be found in Table~\ref{tab:newEg1-L2-combined},
    indicating that both hard and soft truncations achieve the expected order of accuracy.
    However, the numerical error obtained with hard truncation is significantly smaller than that with soft truncation.
    This is because the exact solution is rank 1, i.e.\
    the first singular value is of order $\mathcal{O}(1)$, while all others are zero. Therefore, the hard truncation which does not modify the dominant singular values is more accurate. We note that for all other numerical examples in the paper that are not manufactured solutions, the soft and hard truncation yields similar errors.
    Figure~\ref{fig:newEg1} gives the rank evolution of the numerical solution,
    showing that the hard truncation always keeps the numerical rank at 1,  but the soft truncation may lead to a rank of 2.

\end{example}

\begin{table}[htbp]
    \caption{Example~\ref{newEg1}. $L^2$ errors and rates of convergence}\label{tab:newEg1-L2-combined}
    \centering
    \begin{subtable}{\textwidth}
        \centering
        \begin{tabular}{c|c|c|c|c}
            \hline
            \multicolumn{5}{c}{Hard truncation}                                               \\
            \hline
            $N_t$      & 40       & 80                & 160               & 320               \\
            \hline
            SDC-mBUG-2 & 6.12E-05 & 1.68E-05\, (1.86) & 4.39E-06\, (1.94) & 1.12E-06\, (1.98) \\
            \hline
            SDC-mBUG-3 & 4.89E-07 & 7.63E-08\, (2.68) & 1.05E-08\, (2.86) & 1.43E-09\, (2.88) \\
            \hline
            SDC-mBUG-4 & 7.71E-09 & 1.01E-09\, (2.94) & 1.26E-10\, (3.00) & 5.41E-12\, (4.54) \\
            \hline
            \multicolumn{5}{c}{Soft truncation}                                               \\
            \hline
            $N_t$      & 40       & 80                & 160               & 320               \\
            \hline
            SDC-mBUG-2 & 4.50E-02 & 1.20E-02\, (1.91) & 2.96E-03\, (2.01) & 7.37E-04\, (2.01) \\
            \hline
            SDC-mBUG-3 & 7.97E-03 & 1.01E-03\, (2.98) & 1.30E-04\, (2.97) & 1.66E-05\, (2.96) \\
            \hline
            SDC-mBUG-4 & 1.09E-03 & 6.96E-05\, (3.97) & 4.43E-06\, (3.97) & 2.69E-07\, (4.04) \\
            \hline
        \end{tabular}
    \end{subtable}
\end{table}

\begin{figure}[htbp]
    \centering
    \begin{subfigure}[b]{0.47\textwidth}
        \centering
        \includegraphics[width=\textwidth]{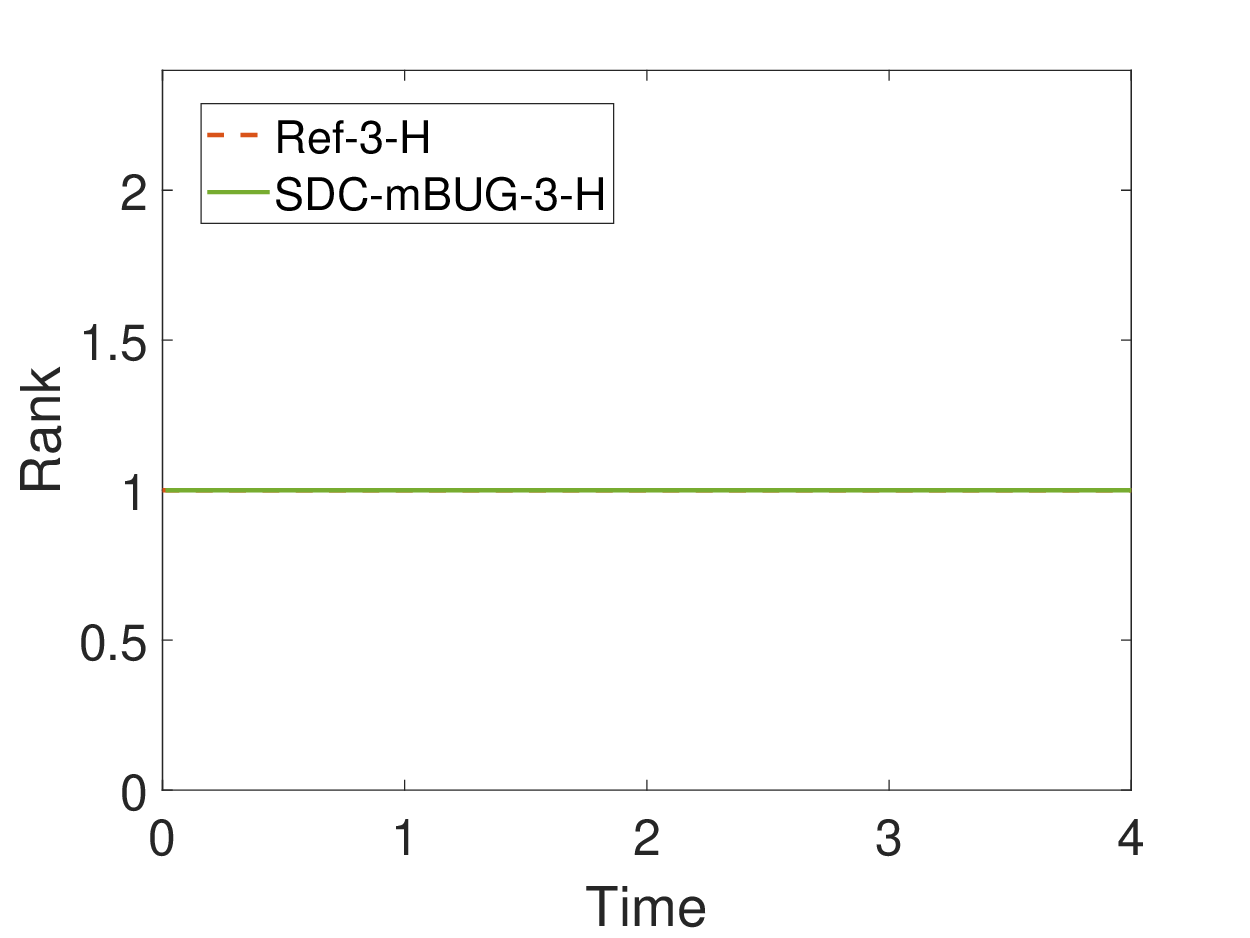}
    \end{subfigure}
    \begin{subfigure}[b]{0.47\textwidth}
        \centering
        \includegraphics[width=\textwidth]{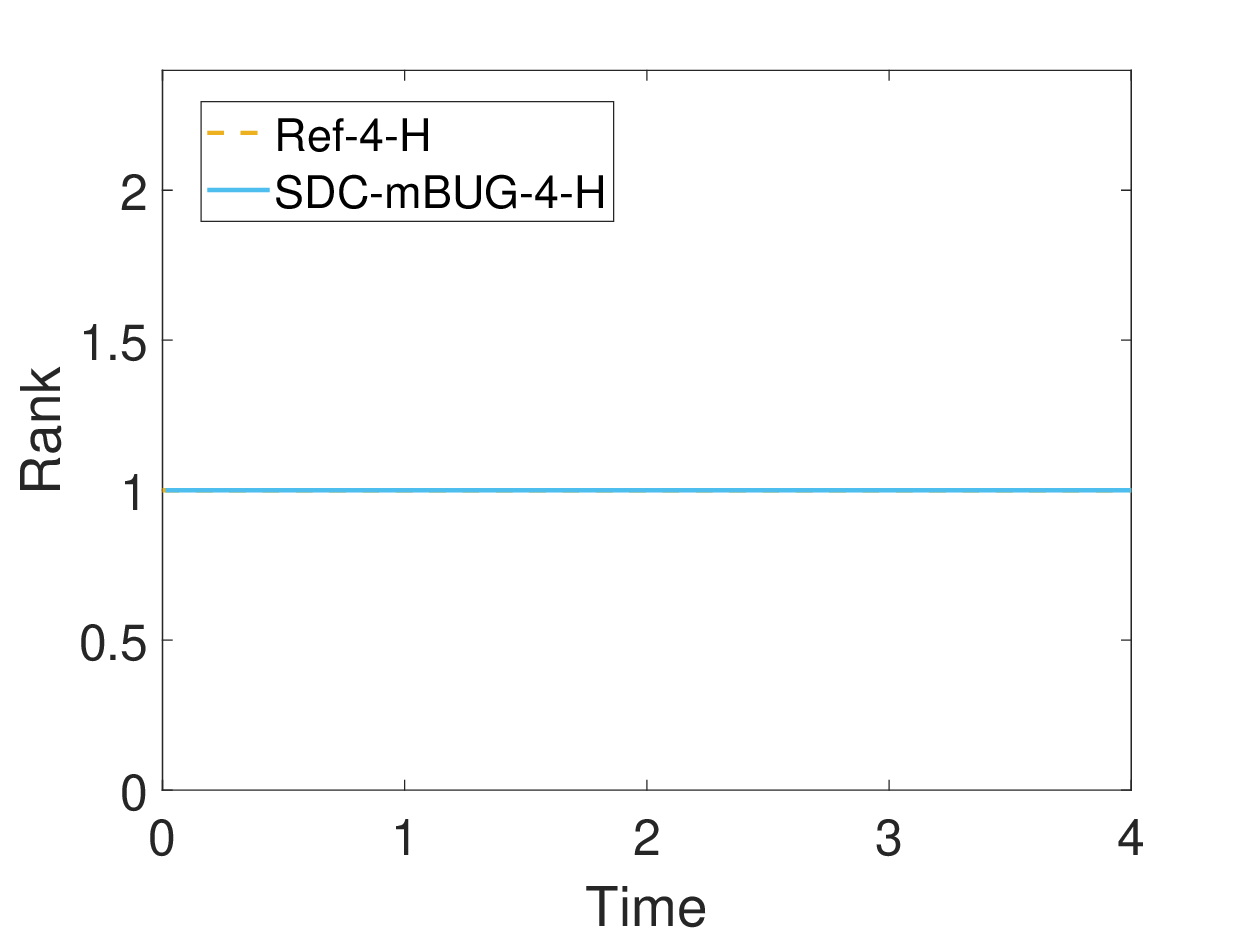}
    \end{subfigure}
    \begin{subfigure}[b]{0.47\textwidth}
        \centering
        \includegraphics[width=\textwidth]{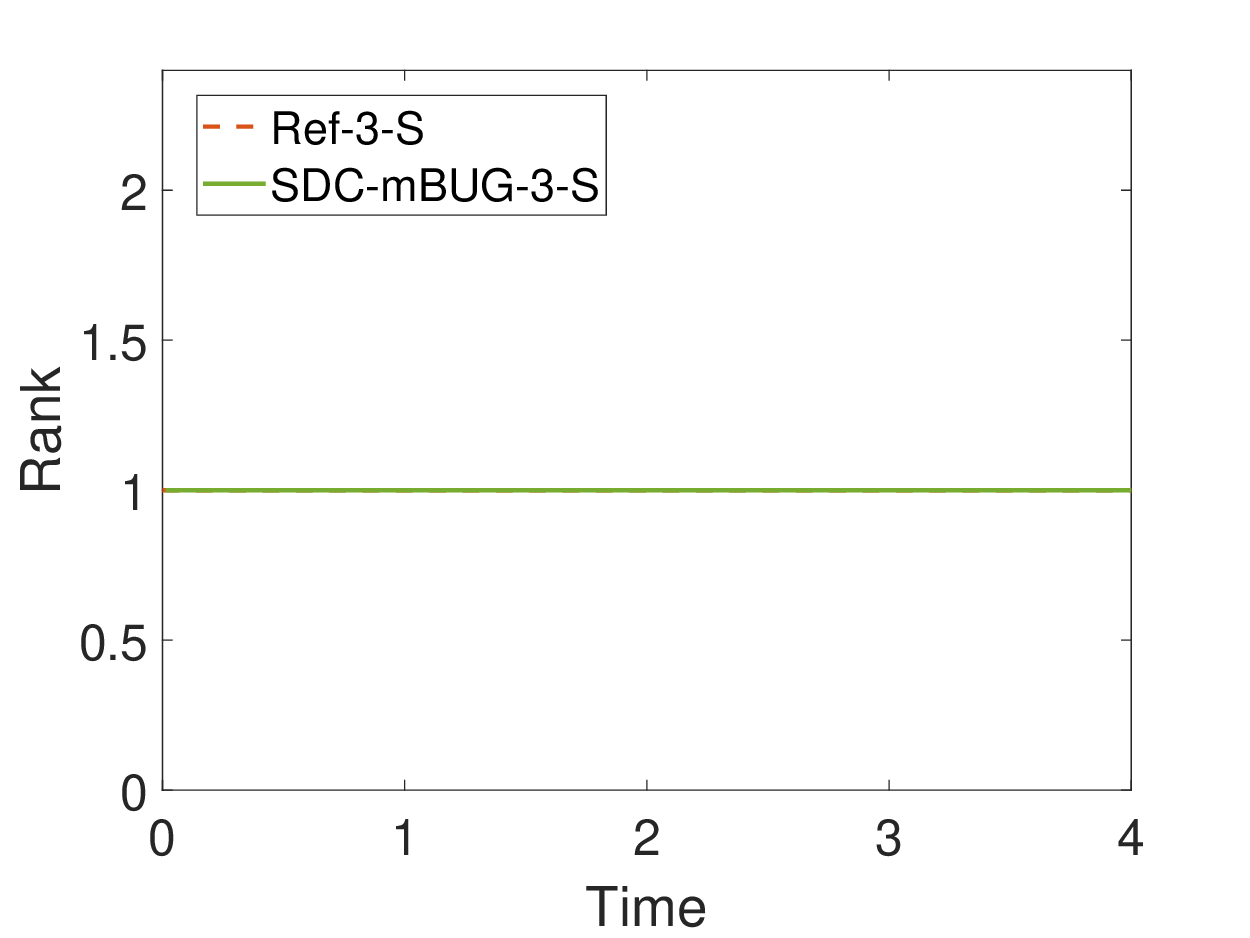}
    \end{subfigure}
    \begin{subfigure}[b]{0.47\textwidth}
        \centering
        \includegraphics[width=\textwidth]{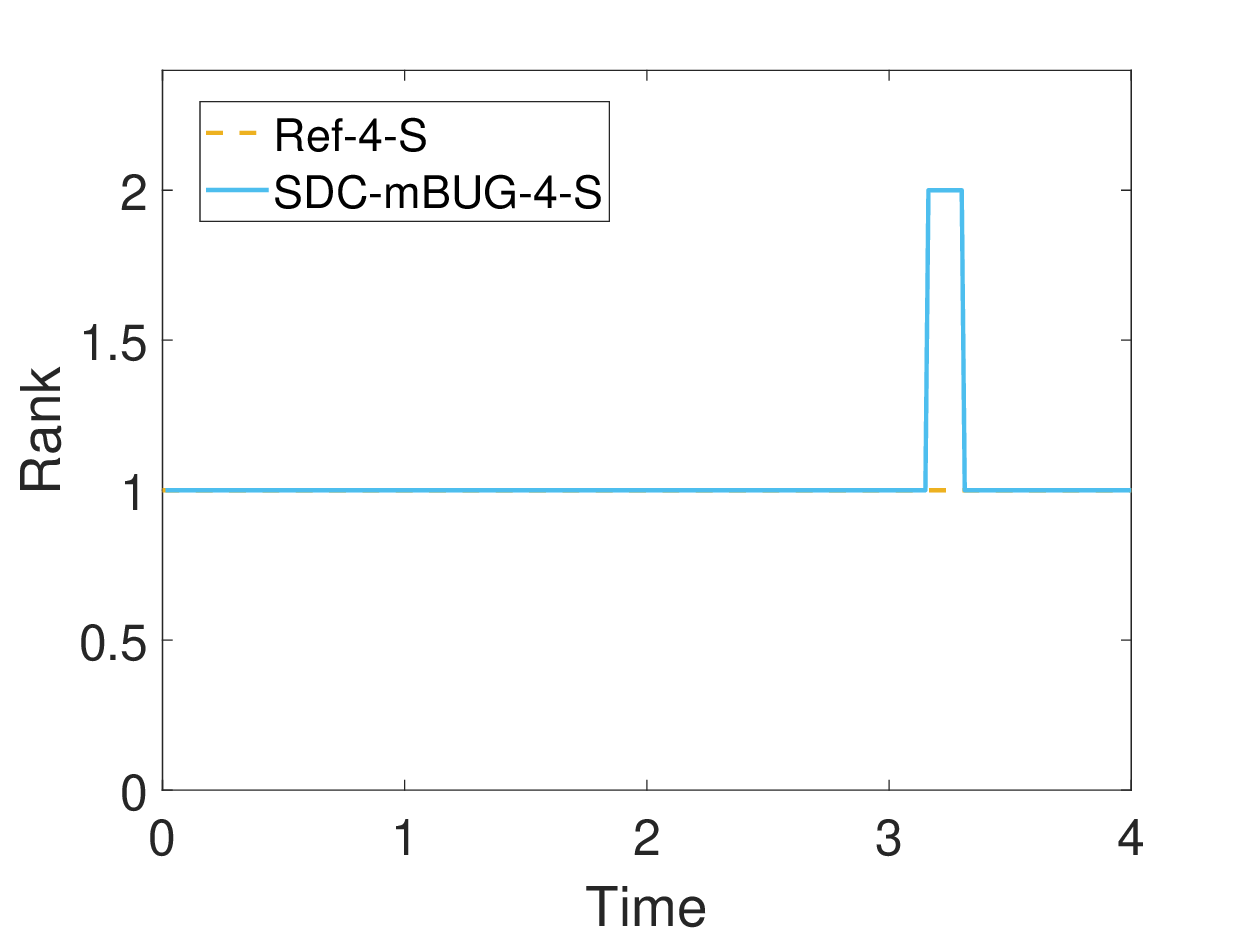}
    \end{subfigure}
    \caption{Example~\ref{newEg1}.
        The rank evolution of the numerical solutions obtained by the SDC-mBUG schemes over time is shown,
        where the dashed line represents the rank evolution of the reference solution.
        Top: hard truncation; bottom: soft truncation.
        Left: the third-order SDC-mBUG scheme;
        right: the fourth-order SDC-mBUG scheme.}
    \label{fig:newEg1}
\end{figure}


\begin{example}[Schrödinger equation]\label{newEg3}
    In this example, we consider a Schrödinger equation as in \cite{kieri2016discretized},
    \[
        \begin{aligned}
            i \partial_t u(x,y,t) & = -\frac12 \nabla^2 u(x,y,t) + \frac12
            \begin{pmatrix}
                x & y
            \end{pmatrix}
            \begin{pmatrix}
                2 & -1 \\ -1 & 3
            \end{pmatrix}
            \begin{pmatrix}
                x \\ y
            \end{pmatrix},
            \\
            u(x,y,0)              & = \frac{1}{\sqrt{\pi}} \exp(\frac12 x^2 + \frac12(y-1)^2).
        \end{aligned}
    \]
    We choose the final time $T=2.0$.
    A reference solution is computed using a mesh $(N_x, N_y, N_t)=(400,400,10000)$.

    In Table~\ref{tab:newEg3-L2-combined}, we display the errors and rates of convergence. Both methods achieve the desired order of accuracy, and the numerical errors between the two are very similar.
    We also track the rank of the solution as a function of time, which can be found in Figure~\ref{fig:newEg3}.
    The rank of the exact solution increases over time, and the rank of the numerical solution follows this trend while maintaining the corresponding accuracy.
\end{example}

\begin{table}[htbp]
    \caption{Example~\ref{newEg3}. $L^2$ errors and rates of convergence}\label{tab:newEg3-L2-combined}
    \centering
    \begin{subtable}{\textwidth}
        \centering
        \begin{tabular}{c|c|c|c|c}
            \hline
            \multicolumn{5}{c}{Hard truncation}                                               \\
            \hline
            $N_t$      & 50       & 100               & 200               & 400               \\
            \hline
            SDC-mBUG-2 & 1.17E-01 & 4.43E-02\, (1.40) & 1.31E-02\, (1.75) & 3.42E-03\, (1.94) \\
            \hline
            SDC-mBUG-3 & 1.10E-02 & 1.62E-03\, (2.76) & 2.09E-04\, (2.96) & 2.63E-05\, (2.99) \\
            \hline
            SDC-mBUG-4 & 8.93E-04 & 5.62E-05\, (3.99) & 3.43E-06\, (4.03) & 2.41E-07\, (3.83) \\
            \hline
            \multicolumn{5}{c}{Soft truncation}                                               \\
            \hline
            $N_t$      & 50       & 100               & 200               & 400               \\
            \hline
            SDC-mBUG-2 & 1.23E-01 & 4.60E-02\, (1.42) & 1.36E-02\, (1.76) & 3.54E-03\, (1.94) \\
            \hline
            SDC-mBUG-3 & 1.23E-02 & 1.81E-03\, (2.76) & 2.29E-04\, (2.98) & 2.85E-05\, (3.01) \\
            \hline
            SDC-mBUG-4 & 1.14E-03 & 7.05E-05\, (4.02) & 4.20E-06\, (4.07) & 2.79E-07\, (3.91) \\
            \hline
        \end{tabular}
    \end{subtable}
\end{table}

\begin{figure}[htbp]
    \centering
    \begin{subfigure}[b]{0.47\textwidth}
        \centering
        \includegraphics[width=\textwidth]{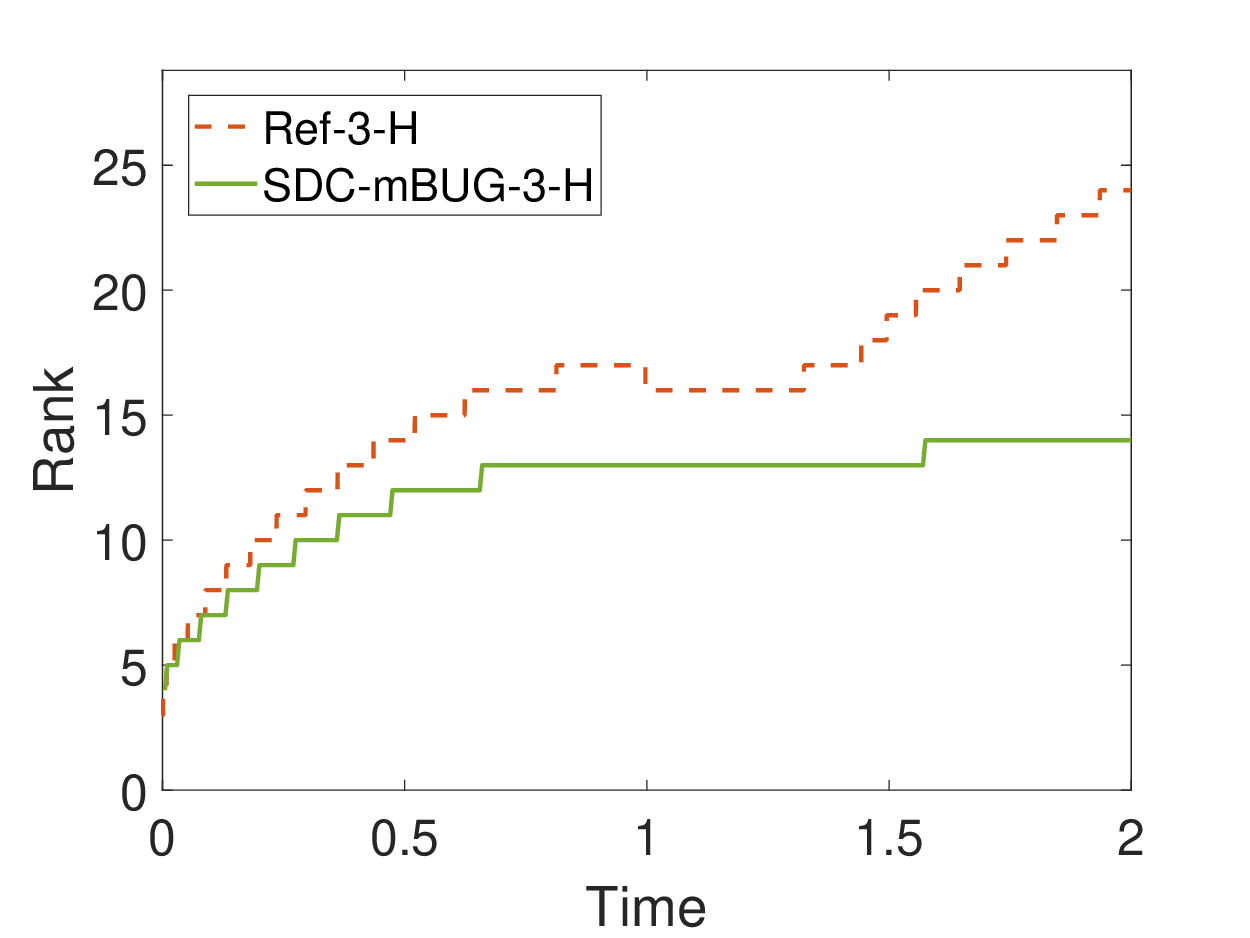}
    \end{subfigure}
    \begin{subfigure}[b]{0.47\textwidth}
        \centering
        \includegraphics[width=\textwidth]{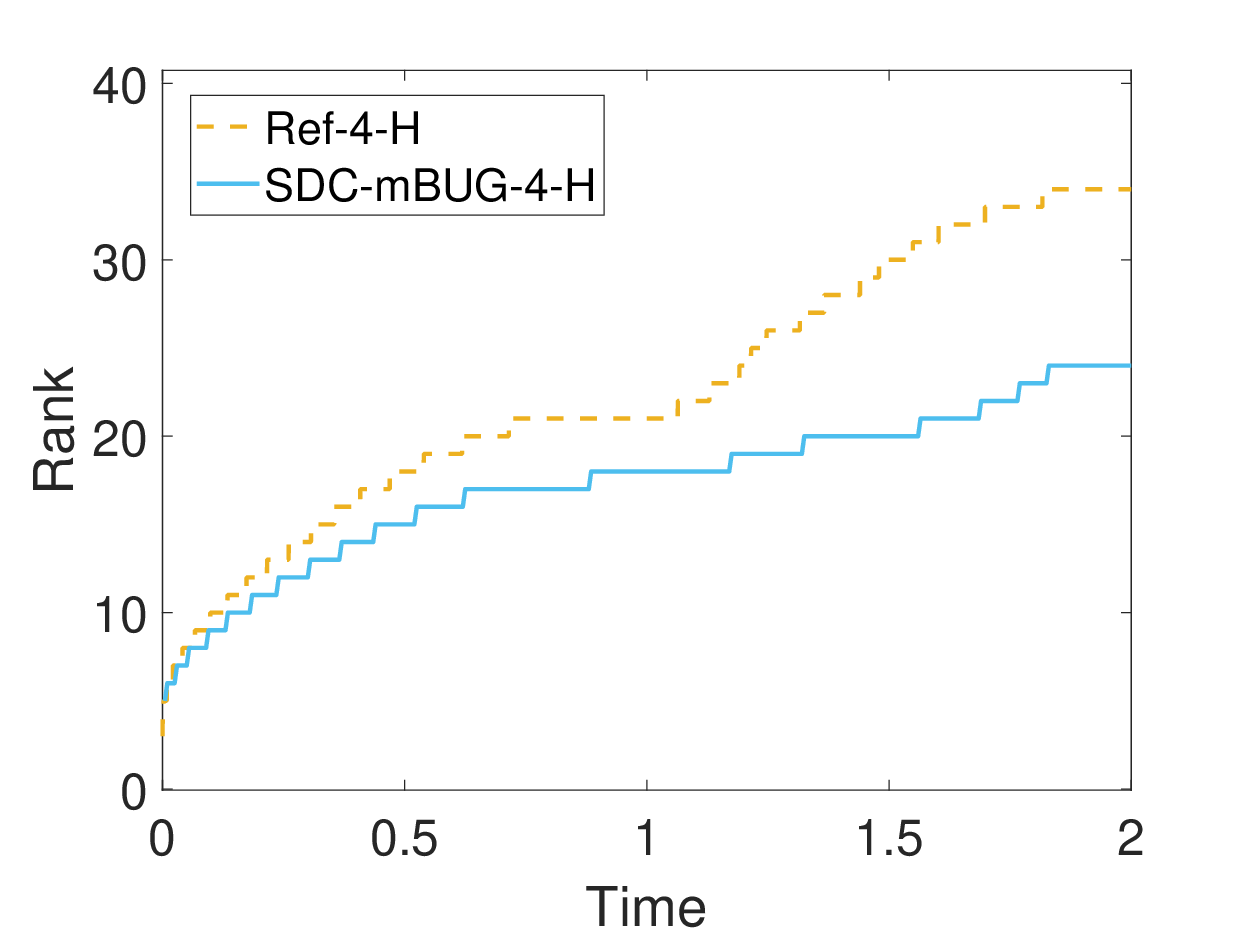}
    \end{subfigure}
    \begin{subfigure}[b]{0.47\textwidth}
        \centering
        \includegraphics[width=\textwidth]{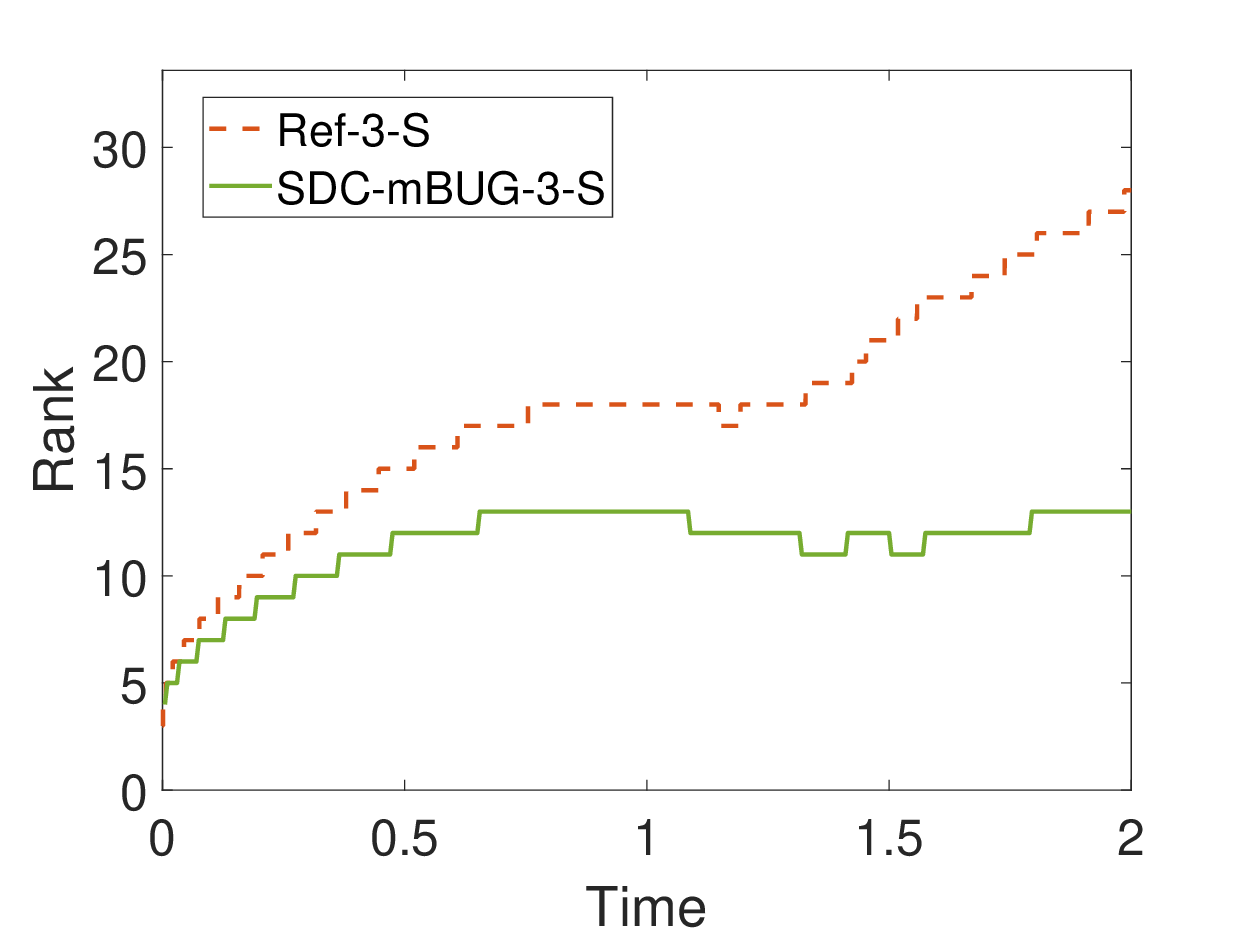}
    \end{subfigure}
    \begin{subfigure}[b]{0.47\textwidth}
        \centering
        \includegraphics[width=\textwidth]{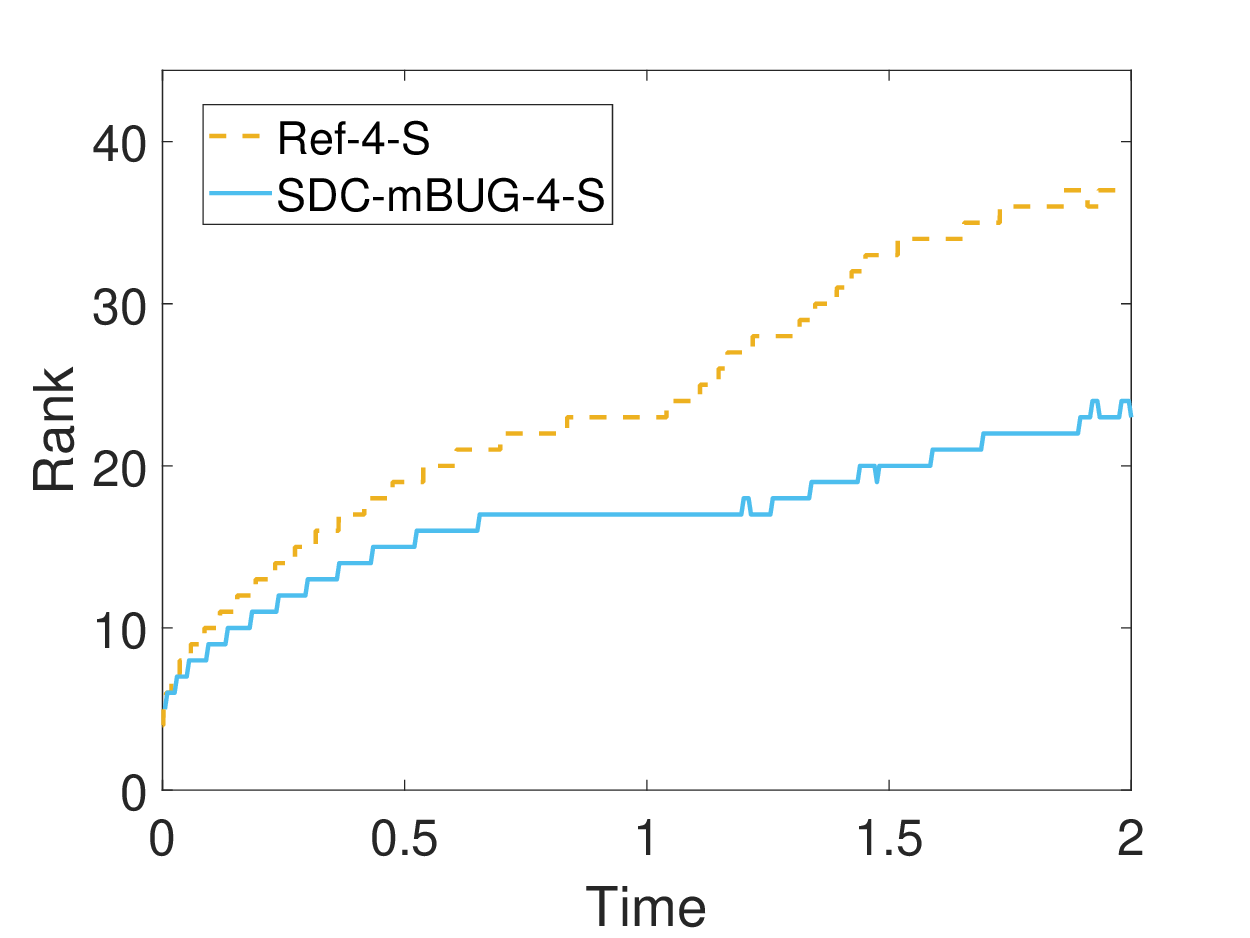}
    \end{subfigure}
    \caption{Example~\ref{newEg3}.
        The rank evolution of the numerical solutions obtained by the SDC-mBUG schemes over time is shown,
        where the dashed line represents the rank evolution of the reference solution.
        Top: hard truncation; bottom: soft truncation.
        Left: the third-order SDC-mBUG scheme;
        right: the fourth-order SDC-mBUG scheme. }
    \label{fig:newEg3}
\end{figure}


\begin{example}[Rotation with anisotropic diffusion]\label{newEg4}
    Next, we consider solid body rotation, with anisotropic diffusion.
    \[
        u_t - y u_x + x u_y = d R(x,y,t,u),
    \]
    where $d = 0.01$,
    \[
        \begin{aligned}
            R(x,y,t,u) ={} & b_1(y)\partial_x\left(a_1(x) \partial_x u\right)
            + b_2(y) \partial_{xy} \left(a_2(x) u\right)                      \\
                           & + a_3(x) \partial_{xy} \left(b_3(y) u\right)
            + a_4(x) \partial_y \left(b_4(y) u\right).
        \end{aligned}
    \]
    The diffusion coefficients are:
    \[
        \begin{aligned}
            a_1(x) & = 1+0.1\sin(0.5 x),    & b_1(y) & = 1+0.1\cos(0.5 x),    \\
            a_2(x) & = 0.15+0.1\sin(0.5 x), & b_2(y) & = 0.15+0.1\cos(0.5 x), \\
            a_3(x) & = 0.15+0.1\cos(0.5 x), & b_3(y) & = 0.15+0.1\sin(0.5 x), \\
            a_4(x) & = 1+0.1\sin(0.5 x),    & b_4(y) & = 1+0.1\cos(0.5 x).
        \end{aligned}
    \]

    We take the rank-1 initial data as
    \[
        u(x,y,0) = \exp(-(x^2+9y^2)),
    \]
    and evolve the solution until time $T=\pi$.
    The full-rank reference solution is computed on a mesh $(N_x, N_y, N_t)=(400,400,12000)$. 
    Table~\ref{tab:newEg4-L2-combined} demonstrates that the SDC-mBUG schemes achieve the corresponding order of accuracy.  Numerical errors from hard and soft truncations are very similar.
    Figure~\ref{fig:newEg4} shows that the numerical solutions always maintain a low-rank structure.
    However, compared to hard truncation, the rank of the results with the soft truncation algorithm fits better than that of the reference solution over time.
    This is because the singular value sequence in the current example exhibits smooth decay.
    Hard truncation introduces sudden changes in the singular value sequence, whereas soft truncation weakens this effect.
\end{example}

\begin{table}[htbp]
    \caption{Example~\ref{newEg4}. $L^2$ errors and rates of convergence.} \label{tab:newEg4-L2-combined}
    \centering
    \begin{subtable}{\textwidth}
        \centering
        \begin{tabular}{c|c|c|c|c}
            \hline
            \multicolumn{5}{c}{Hard truncation}                                               \\
            \hline
            $N_t$      & 100      & 200               & 400               & 600               \\
            \hline
            SDC-mBUG-2 & 4.09E-03 & 1.06E-03\, (1.95) & 2.71E-04\, (1.97) & 1.19E-04\, (2.02) \\
            \hline
            SDC-mBUG-3 & 1.27E-04 & 1.36E-05\, (3.22) & 1.77E-06\, (2.94) & 5.84E-07\, (2.74) \\
            \hline
            SDC-mBUG-4 & 3.71E-06 & 2.63E-07\, (3.82) & 1.78E-08\, (3.88) & 4.08E-09\, (3.64) \\
            \hline
            \multicolumn{5}{c}{Soft truncation}                                               \\
            \hline
            $N_t$      & 100      & 200               & 400               & 600               \\
            \hline
            SDC-mBUG-2 & 8.85E-03 & 2.26E-03\, (1.97) & 5.55E-04\, (2.03) & 2.45E-04\, (2.01) \\
            \hline
            SDC-mBUG-3 & 3.54E-04 & 4.14E-05\, (3.09) & 4.93E-06\, (3.07) & 1.44E-06\, (3.04) \\
            \hline
            SDC-mBUG-4 & 1.20E-05 & 6.57E-07\, (4.19) & 3.92E-08\, (4.07) & 7.78E-09\, (3.99) \\
            \hline
        \end{tabular}
    \end{subtable}
\end{table}

\begin{figure}[htbp]
    \centering
    \begin{subfigure}[b]{0.47\textwidth}
        \centering
        \includegraphics[width=\textwidth]{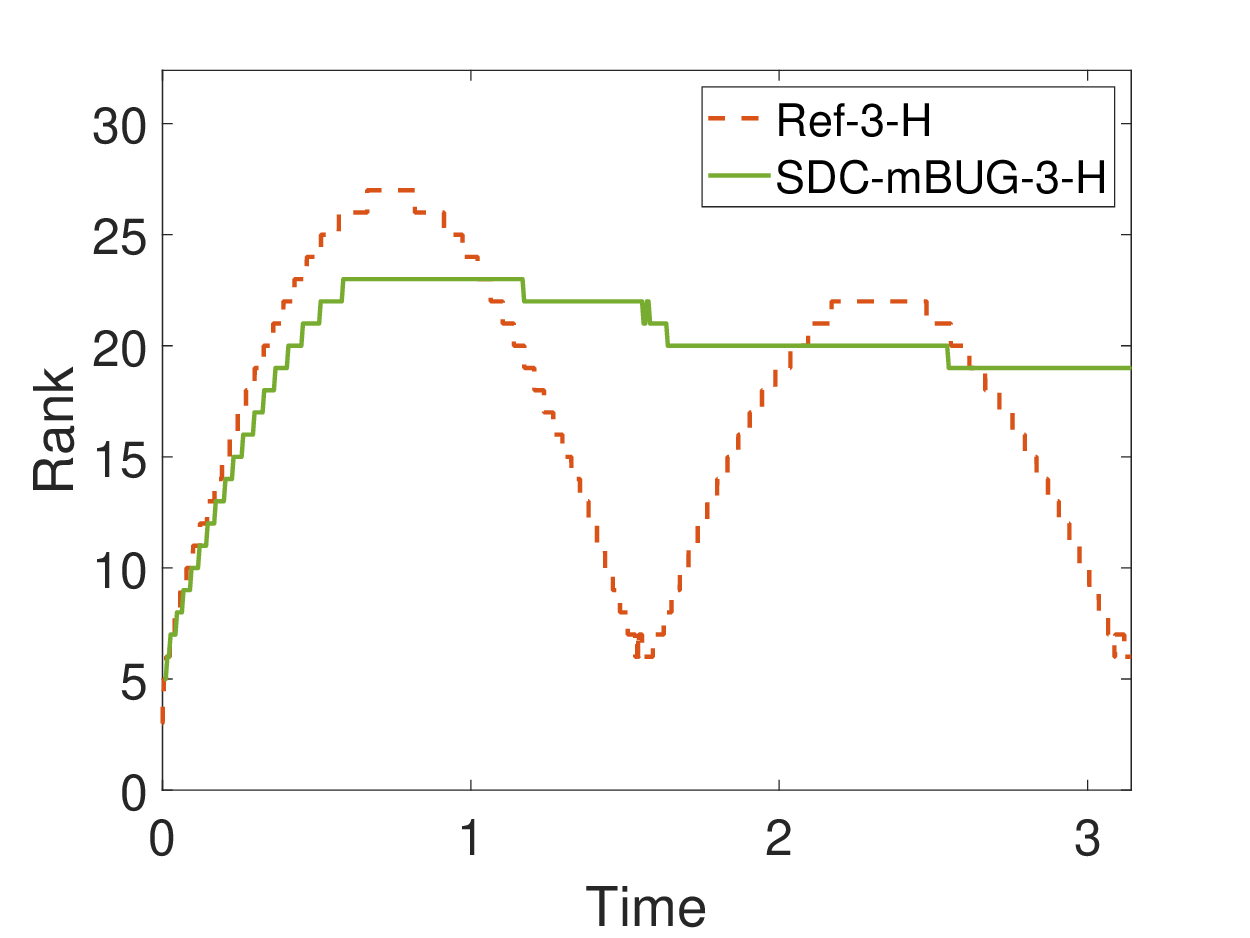}
    \end{subfigure}
    \begin{subfigure}[b]{0.47\textwidth}
        \centering
        \includegraphics[width=\textwidth]{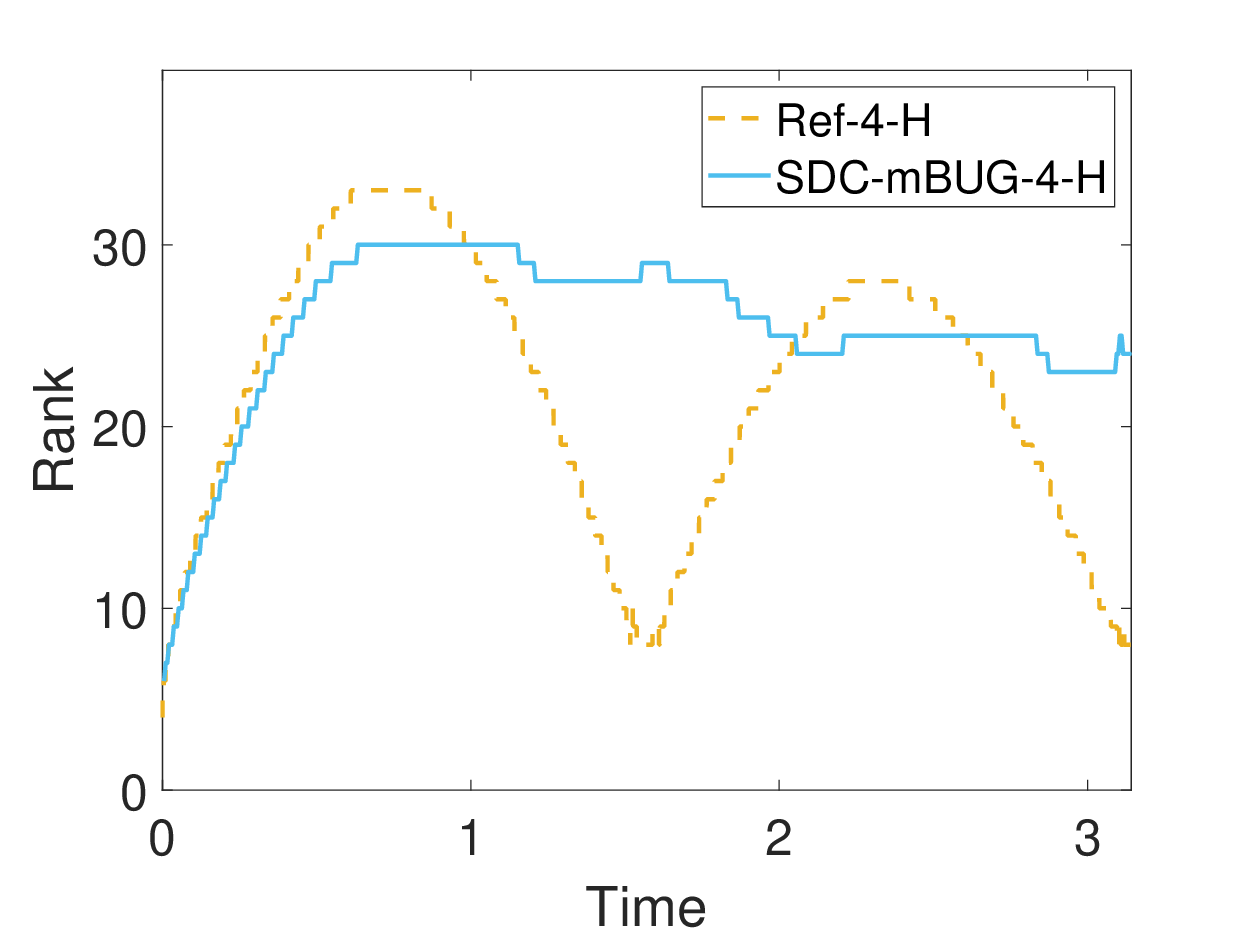}
    \end{subfigure}
    \begin{subfigure}[b]{0.47\textwidth}
        \centering
        \includegraphics[width=\textwidth]{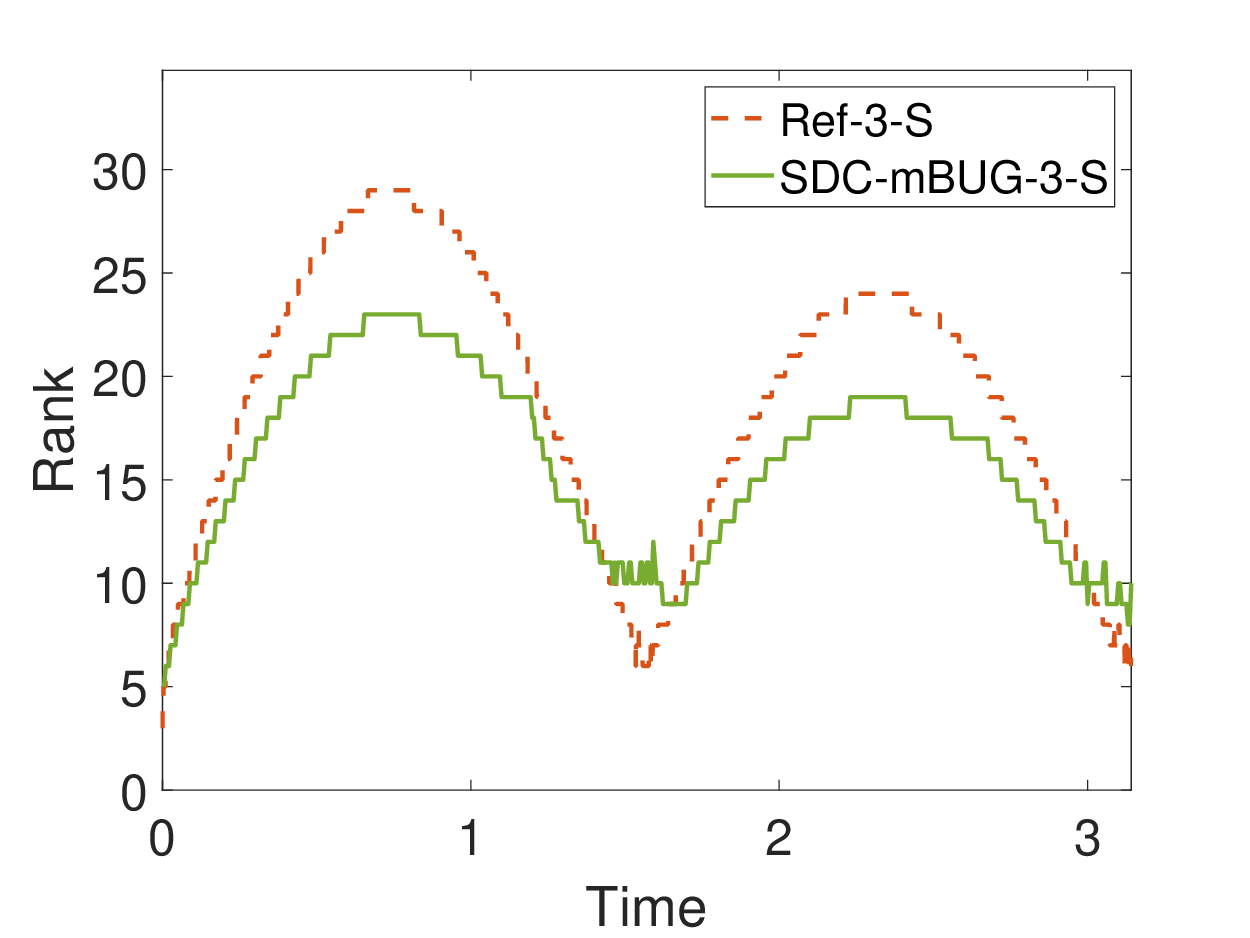}
    \end{subfigure}
    \begin{subfigure}[b]{0.47\textwidth}
        \centering
        \includegraphics[width=\textwidth]{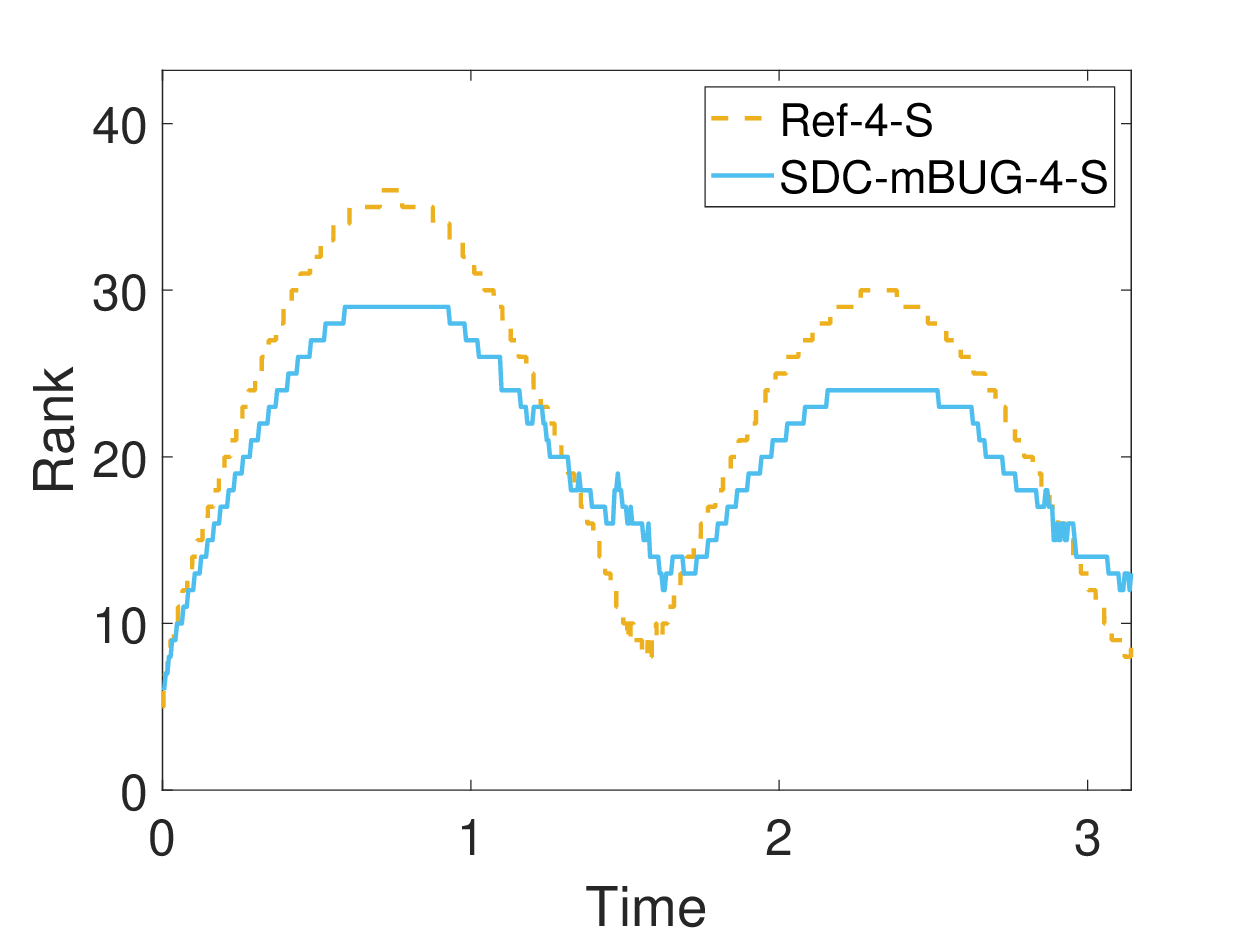}
    \end{subfigure}
    \caption{Example~\ref{newEg4}.
        The rank evolution of the numerical solutions obtained by the SDC-mBUG schemes over time is shown,
        where the dashed line represents the rank evolution of the reference solution.
        Top: hard truncation; bottom: soft truncation.
        Left: the third-order SDC-mBUG scheme;
        right: the fourth-order SDC-mBUG scheme. }
    \label{fig:newEg4}
\end{figure}


\begin{example}[Pure rotation]\label{newEg5}
    In this example, we consider rigid body rotation without diffusion
    \[
        u_t - y u_x + x u_y = 0.
    \]
    The exact solution can be obtained by rotating the initial value at a fixed angular velocity,
    resulting in the evolution of the rank of the exact solution showing strict periodicity.
    We take the initial data as
    \[
        u(x,y,0) = \exp(-(5x^2+5y^2+8xy)).
    \]

    We evolve the solution until time $T=\pi$.
    The errors and rates of convergence can be found in Table~\ref{tab:newEg5-L2-combined}.
    The rank of the solutions is shown in Figure~\ref{fig:newEg5}. 
    Similar to the example above, the SDC-mBUG scheme achieves the corresponding accuracy order with similar errors.
    The rank evolution of the numerical solution under hard truncation and soft truncation is also significantly different.
    For the soft truncation, the rank evolves in a periodic pattern, while the rank curve with hard truncation increases over time.
\end{example}

\begin{table}[htbp]
    \caption{Example~\ref{newEg5}. $L^2$ errors and rates of convergence.} \label{tab:newEg5-L2-combined}
    \centering
    \begin{subtable}{\textwidth}
        \centering
        \begin{tabular}{c|c|c|c|c}
            \hline
            \multicolumn{5}{c}{Hard truncation}                                               \\
            \hline
            $N_t$      & 100      & 200               & 400               & 600               \\
            \hline
            SDC-mBUG-2 & 1.34E-02 & 3.92E-03\, (1.77) & 1.04E-03\, (1.92) & 4.59E-04\, (2.01) \\
            \hline
            SDC-mBUG-3 & 4.51E-04 & 7.59E-05\, (2.57) & 1.04E-05\, (2.87) & 3.11E-06\, (2.97) \\
            \hline
            SDC-mBUG-4 & 2.21E-05 & 1.74E-06\, (3.67) & 1.37E-07\, (3.67) & 2.68E-08\, (4.02) \\
            \hline
            \multicolumn{5}{c}{Soft truncation}                                               \\
            \hline
            $N_t$      & 100      & 200               & 400               & 600               \\
            \hline
            SDC-mBUG-2 & 2.00E-02 & 5.60E-03\, (1.84) & 1.43E-03\, (1.97) & 6.37E-04\, (1.99) \\
            \hline
            SDC-mBUG-3 & 1.01E-03 & 1.26E-04\, (3.00) & 1.52E-05\, (3.05) & 4.51E-06\, (3.00) \\
            \hline
            SDC-mBUG-4 & 4.94E-05 & 3.13E-06\, (3.98) & 2.02E-07\, (3.95) & 4.16E-08\, (3.90) \\
            \hline
        \end{tabular}
    \end{subtable}
\end{table}

\begin{figure}[htbp]
    \centering
    \begin{subfigure}[b]{0.47\textwidth}
        \centering
        \includegraphics[width=\textwidth]{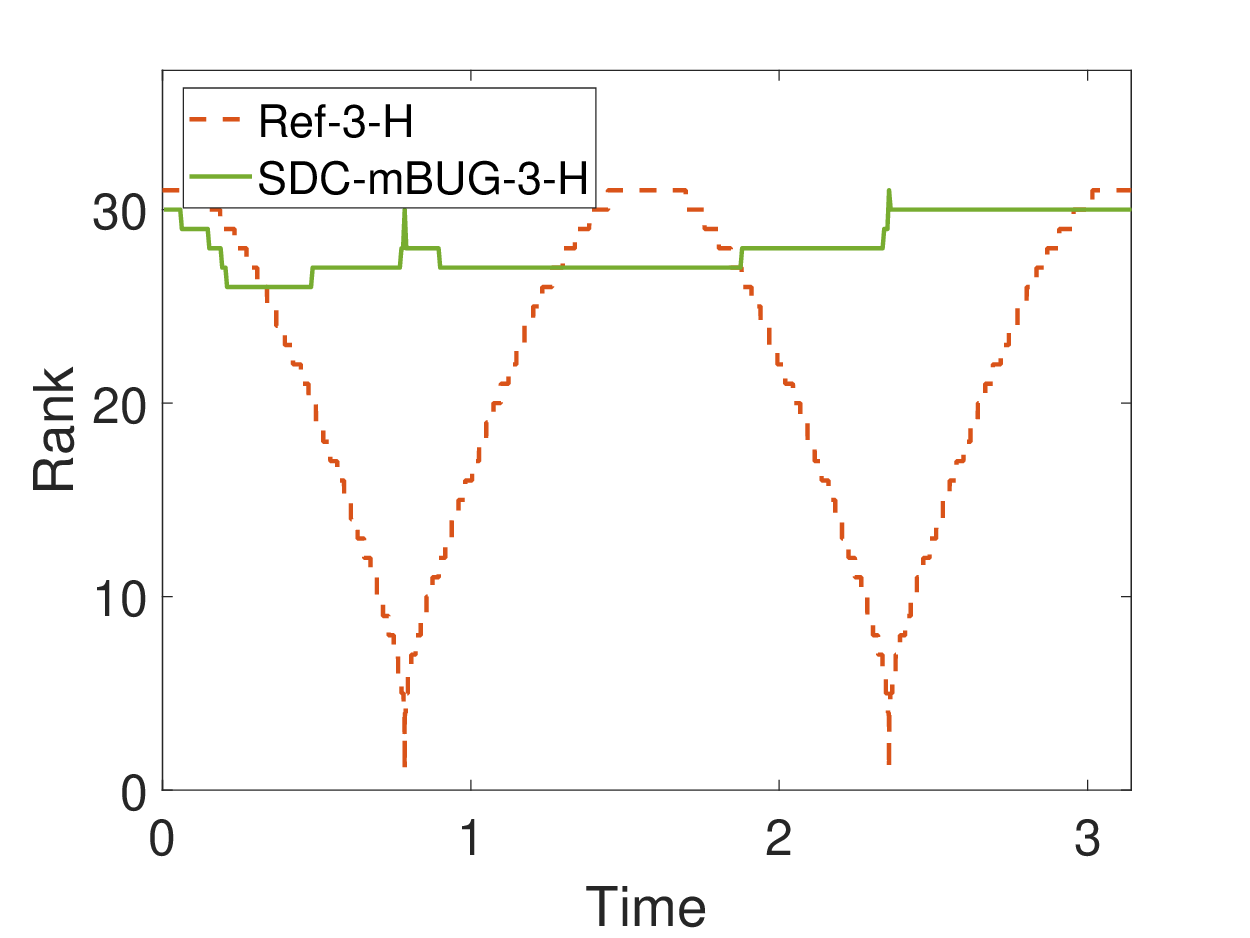}
    \end{subfigure}
    \begin{subfigure}[b]{0.47\textwidth}
        \centering
        \includegraphics[width=\textwidth]{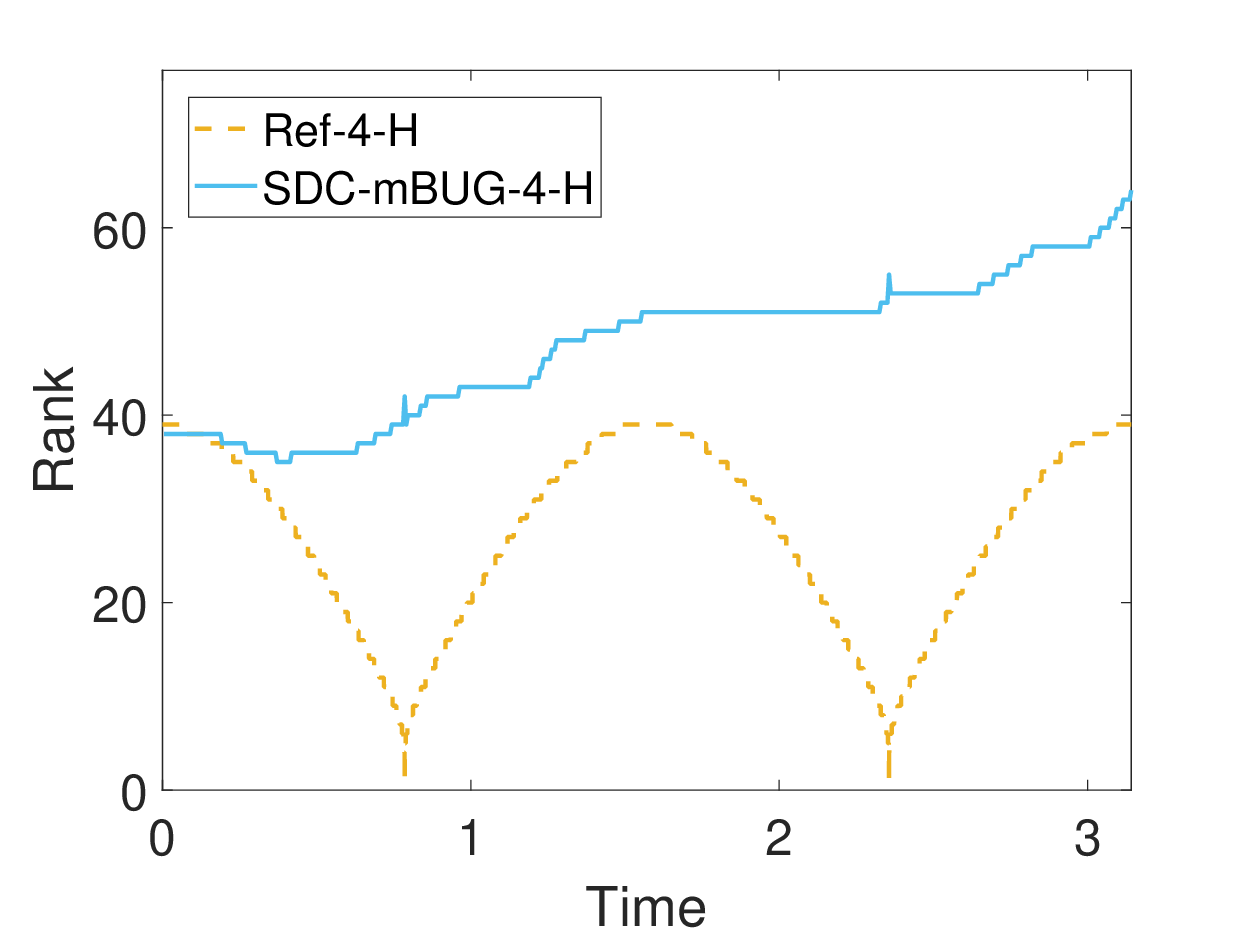}
    \end{subfigure}
    \begin{subfigure}[b]{0.47\textwidth}
        \centering
        \includegraphics[width=\textwidth]{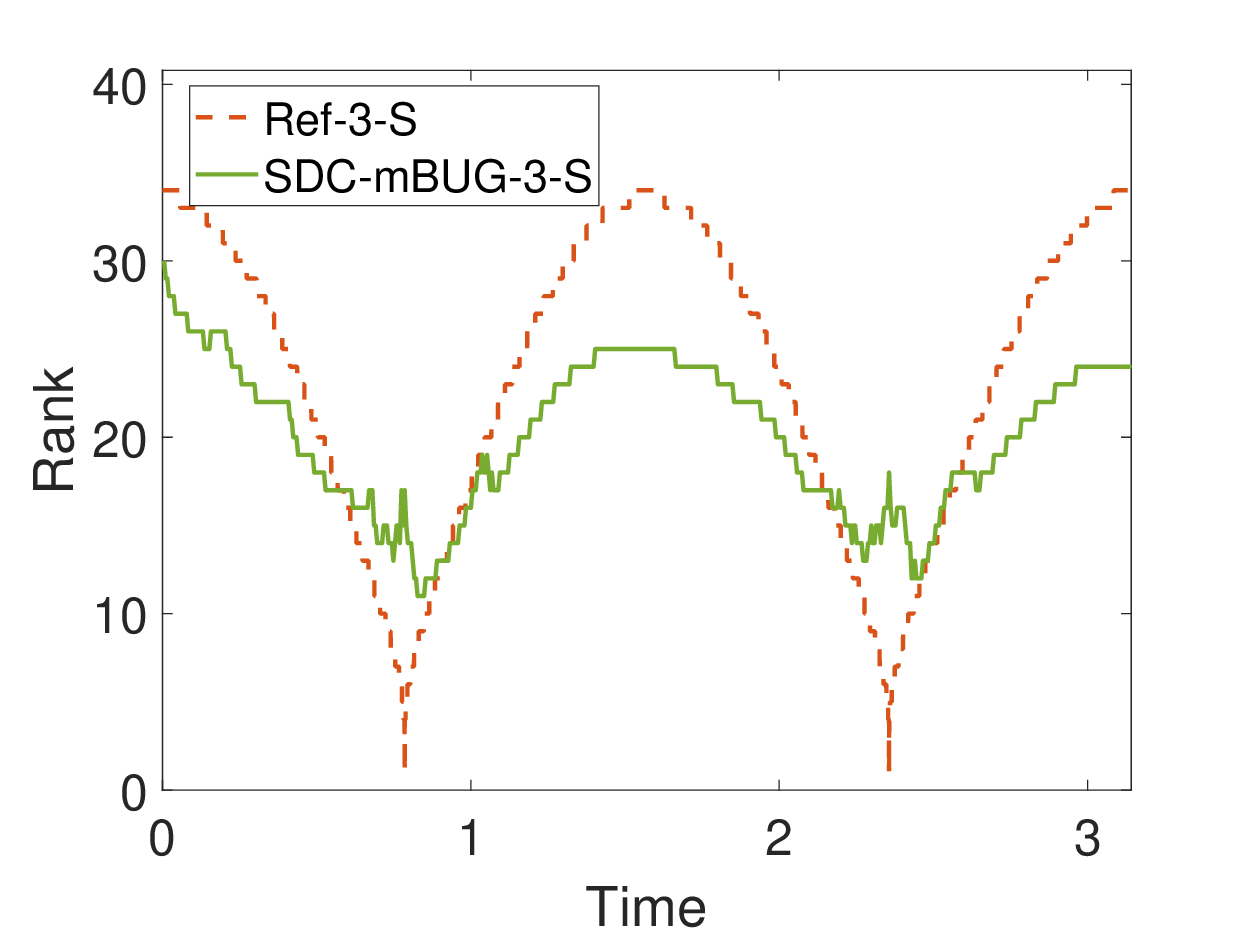}
    \end{subfigure}
    \begin{subfigure}[b]{0.47\textwidth}
        \centering
        \includegraphics[width=\textwidth]{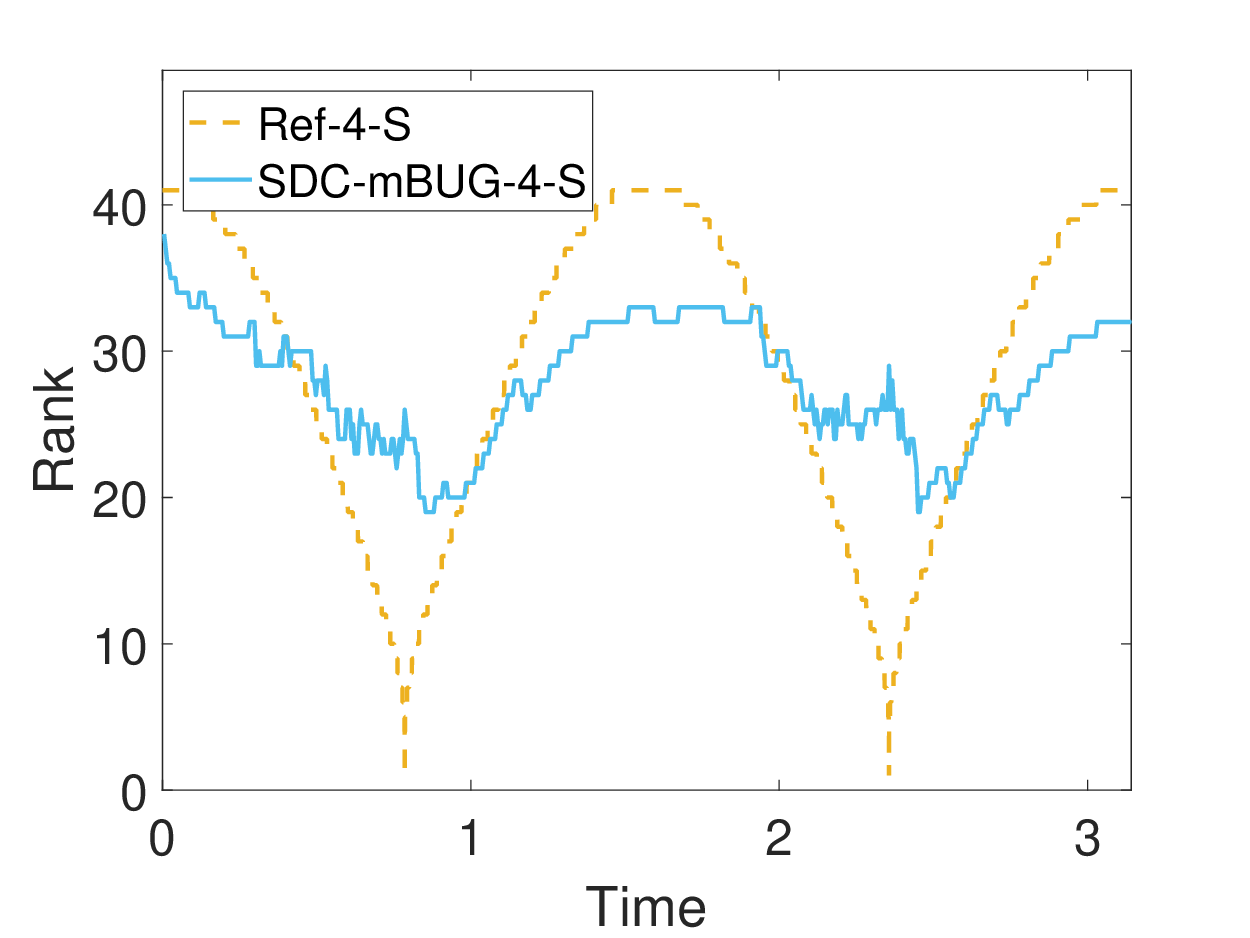}
    \end{subfigure}
    \caption{Example~\ref{newEg5}.
        The rank evolution of the numerical solutions obtained by the SDC-mBUG schemes over time is shown,
        where the dashed line represents the rank evolution of the reference solution.
        Top: hard truncation; bottom: soft truncation.
        Left: the third-order SDC-mBUG scheme;
        right: the fourth-order SDC-mBUG scheme. }
    \label{fig:newEg5}
\end{figure}


\section{Conclusions and future work}\label{sec:conclusions-and-future-work}

In this work, we design a high order low-rank scheme based on SDC with mBUG, which is a modified dynamic low-rank approximation.
By integrating SDC with the adaptive rank truncation strategy,
the method can achieve high temporal accuracy and rank control. We provide truncation error analysis which shows the convergence order and also motivates with choice of tolerance parameters across levels. In particular, to reduce the computational cost, the correction steps in SDC do not need to compute $K$- and $L$-steps.
We consider two types of truncation strategies: the hard truncation and the soft truncation.
Numerical results show that while both truncation strategies achieve the desired order of accuracy, there is a difference in numerical  error and rank control.
For problems with solutions with non-smoothly decaying singular value sequences, the hard truncation can yield more accurate approximations,
whereas the soft truncation tends to offer better rank control for problems with a smoothly decaying singular value sequence.
Future work will focus on further generalization to multidimensional cases and other truncation strategies.

\section*{Ethics declarations}
\noindent\textbf{Conflict of interest} \\
The authors declare that they have no conflict of interest.

\section*{Data Availability}
\noindent Data sets generated during the current study are available from the corresponding author on reasonable request.

\bibliographystyle{abbrv}
\bibliography{reference}

\end{document}